\documentclass[12pt]{amsart} 
\usepackage{amscd}
\usepackage{amssymb}
\usepackage{a4wide}
\usepackage{amstext}
\usepackage{amsthm}
\usepackage{xcolor}
\usepackage{cite}
\usepackage[T1,T2A]{fontenc}
\usepackage[utf8]{inputenc}
\usepackage{url}
\usepackage{amsfonts}
\usepackage{amssymb, amsthm}
\usepackage{amsmath}
\usepackage{mathtools}
\usepackage{needspace}
\usepackage[pdftex]{graphicx}
\usepackage{hyperref}
\usepackage{datetime}
\usepackage{epigraph}
\usepackage{verbatim}
\usepackage{mathtools}
\usepackage{xcolor}
\numberwithin{equation}{section}

\newcommand{\Z}{\mathbb{Z}}

\newcommand{\N}{\mathbb{N}}
\newcommand{\R}{\mathbb{R}}

\newcommand{\Cm}{\mathbb{C}}

\newcommand{\eps}{\varepsilon}

\newcommand{\abs}[1]{| #1 |}

\newcommand{\norm}[1]{\left\Vert #1 \right\Vert}

\DeclareMathOperator{\beqq}{\begin{equation}} 

\DeclareMathOperator{\eeqq}{\end{equation}}

\renewcommand{\phi}{\varphi}

\newcommand{\beq}{\begin{equation}}
\newcommand{\eeq}{\end{equation}}

\newcommand{\F}{\mathcal{F}}

\newtheorem{Thm}{Theorem}[section]
\newtheorem{theorem}[Thm]{Theorem}
\newtheorem{example}{Example}
\newtheorem{lemma}[Thm]{Lemma}
\newtheorem{proposition}[Thm]{Proposition}
\newtheorem{corollary}[Thm]{Corollary}
\newtheorem{remark}[Thm]{Remark}

\newtheorem{claim}[Thm]{Claim}
\newtheorem{definition}{Definition}

\usepackage{amsrefs}

\makeatletter
\renewcommand{\eprint}[1]{\href{https://arxiv.org/abs/#1}{arXiv:#1}}

\makeatletter
\BibSpec{article}{%
    +{}  {\PrintAuthors}                {author}
    +{,} { \textit}                     {title}
    +{.} { }                            {part}
    +{:} { \textit}                     {subtitle}
    +{,} { \PrintContributions}         {contribution}
    +{.} { \PrintPartials}              {partial}
    +{,} { }                            {journal}
    +{}  { \textbf}                     {volume}
    +{}  { \PrintDatePV}                {date}
    +{,} { \issuetext}                  {number}
    +{,} { \eprintpages}                {pages}
    +{,} { }                            {status}
    +{,} { \PrintDOI}                   {doi}
    +{,} { \eprint}                     {eprint}
    +{}  { \parenthesize}               {language}
    +{}  { \PrintTranslation}           {translation}
    +{;} { \PrintReprint}               {reprint}
    +{.} { }                            {note}
    +{.} {}                             {transition}
}
\makeatother

\makeatletter
\@ifundefined{ReviewList}{}{\renewcommand{\ReviewList}[1]{}}
\@ifundefined{PrintReviews}{}{\renewcommand{\PrintReviews}[1]{}}
\makeatother

\title{Sharp estimates for eigenvalues of localization operators with applications to area laws}

\begin{document}
\sloppy

\author{Aleksei Kulikov}
\address{Aleksei Kulikov,
\newline University of Copenhagen, Department of Mathematical Sciences,
Universitetsparken 5, 2100 Copenhagen, Denmark,
\newline {\tt lyosha.kulikov@mail.ru} 
}

\author{Martin Dam Larsen}
\address{Martin Dam Larsen,
\newline University of Copenhagen, Department of Mathematical Sciences,
Universitetsparken 5, 2100 Copenhagen, Denmark,
\newline {\tt mdl@math.ku.dk} 
}

\begin{abstract} 
We study the eigenvalues of the localization operator $S_{A, B} = P_A\mathcal{F}^{-1}P_B\mathcal{F} P_A$, where $\mathcal{F}$ is the Fourier transform and $A = cA_0, B = B_0$ for some fixed sets $A_0, B_0\subset \mathbb{R}^d$ and a large parameter $c > 0$. For the counting function of the eigenvalues $|\{n: \varepsilon < \lambda_n(A,B)\le 1-\varepsilon\}|$ we obtain  a sharp uniform upper bound if one of the sets is a finite disjoint union of parallelepipeds and a bound which is only a single logarithm off the conjectural optimal bound in the general case. These bounds are applied to the estimation of traces ${\rm{Tr}}\, f(S_{A,B})$ for functions $f$ with a very low regularity, in particular establishing an enhanced area law in the former case.
\end{abstract}

\maketitle
\pagestyle{plain}
\section{Introduction}
\subsection{Time-frequency localization operator}
For a measurable set $A\subset\R^d$ by $P_A$ we denote the projection onto $A$ and by $Q_A$ we denote the Fourier projection onto $A$
\[
Q_A = \F^{-1}P_A\F,
\]
where $\F$ is the Fourier transform
\[
\mathcal{F}f(\xi) = \hat{f}(\xi) = \int_{\R^d} e^{-2\pi i x\cdot \xi} f(x) \, dx,
\]
whose inverse is
$$\F^{-1}f(x) = \check{f}(x) =\int_{\R^d} e^{2\pi i x\cdot \xi}f(\xi)\, d\xi.$$

For a pair of measurable sets $A, B\subset\R^d$ by $S_{A,B}$ we denote the time-frequency localization operator
\[
S_{A,B} = P_AQ_BP_A.
\]
This is a non-negative definite self-adjoint operator on $L^2(\R^d)$ with the integral kernel
\begin{equation}\label{kernel}
S_{A,B}(x,y) = 1_{A}(x) \check{1}_B(x-y) 1_A(y), \quad x,y\in \R^d.
\end{equation}
If the measures of $A$ and $B$ are finite, $S_{A,B}$ is a trace class operator with ${\rm{Tr}}(S_{A,B}) = |A| |B|$. In particular, it is compact and as such it has a sequence of eigenvalues
$$1 > \|S_{A,B}\| =\lambda_1(A,B) \ge \lambda_2(A,B)\ge\ldots \ge 0.$$
We note for future reference that $S_{A,B}$ obviously does not change if we modify $A$ and $B$ by sets of measure zero.

It can be shown that the eigenvalues remain the same if we swap $A$ and $B$, that is $\lambda_n(A,B)=\lambda_n(B,A)$. The fact that $\lambda_1(A,B)$ is strictly less than $1$ says that there does not exist a non-zero function such that it and its Fourier transform have supports of finite measure. This is known as the uncertainty principle of Benedicks and Amrein--Berthier \cites{Amrein_Berthier77,Benedicks85}. In this paper we will be studying the finer properties of the distribution of these eigenvalues.

The usual regime that is considered is when the sets $A$ and $B$ are scalings of some fixed sets $A_0, B_0$. So, we will assume that $A = cA_0$ and $B = B_0$ for a large parameter $c$. Note that we do not put a scaling onto $B$, as it can be transferred to the scaling on $A$ by an affine change of variables without altering the eigenvalues. {More general affine changes of variables will play a role in our proofs later.}

The goal of the present paper is to study the distribution of eigenvalues $\lambda_n(cA_0, B_0)$. We expect that the eigenvalues exhibit a phase transition: the first $\approx c^d|A_0||B_0|$ eigenvalues are very close to $1$, then there are only $\asymp c^{d-1}\log c$ intermediate eigenvalues away from $0$ and $1$, and after that the eigenvalues decay to $0$ extremely fast. The intermediate part of the spectrum is usually referred to as the plunge region. This picture is particularly clear in dimension $1$ in the case when $A_0$ and $B_0$ are fixed intervals, where the exact limiting behaviour and very precise uniform estimates of the eigenvalues are known \cites{Karnik_Romberg_Davenport,widom_landau,kulikov_2026,Bonami_Karoui_17,Fuchs}, see also Section 2 below.

In higher dimensions much less is known as the geometry of sets $A_0, B_0$ starts to play a significant role. To describe the results, let us introduce the counting functions for the eigenvalues.
\begin{equation}\label{counting functions}
\begin{aligned}
N_\eps(A, B) &= |\{n: \lambda_n(c) > \eps\}|,\\
\Lambda^+_\eps(A,B) &= \left|\left\{n: 1 - \eps > \lambda_n(c) > \frac{1}{2}\right\}\right|,\\
\Lambda^-_\eps(A,B) &= \left|\left\{n: \frac{1}{2} \ge \lambda_n(c) > \eps\right\}\right|,\\
\Lambda_\eps(A,B) &= \Lambda^+_\eps(A,B)+\Lambda^-_\eps(A,B) = |\{n: 1 - \eps \ge \lambda_n(c) > \eps\}|.
\end{aligned}
\end{equation} 
We chose this arrangement of strict and non-strict inequalities to ensure that for $0 < \eps < \frac{1}{2}$ we have
\begin{equation}\label{N and Lambda}
N_{1-\eps}(A,B) = N_{1/2}(A,B) - \Lambda_\eps^+(A,B), \quad N_\eps(A,B) = N_{1/2}(A,B) + \Lambda_\eps^-(A,B).
\end{equation}
Of course, for all applications whether we have strict or non-strict inequalities is immaterial.

Previous results established bounds on $N_\eps(cA_0, B_0)$ and $\Lambda_\eps(cA_0,B_0)$ under strong enough assumptions on the boundaries $\partial A_0$ and $\partial B_0$. In particular, Sobolev \cite{Sobolev13} established a two-term asymptotic formula for $N_\eps(cA_0,B_0)$ for fixed $\eps > 0$. There are also several uniform estimates on $N_\eps(cA_0, B_0)$ and $\Lambda_\eps(cA_0, B_0)$ with $\eps \to 0$ as $c\to \infty$ \cites{Israel1,Israel2,bad_combinatorics, Marceca}.

In this work we establish estimates on $\Lambda_\eps(cA_0, B_0)$ which are on the one hand better than all of the previously obtained estimates, in particular being sharp in some cases, and on the other hand requiring much weaker and more natural geometric conditions on the sets $A_0, B_0$.

The form of our results depends on the assumptions we put on the sets $A_0$, $B_0$. If both of them are axis-parallel boxes then we can leverage known estimates in the one-dimensional case to establish precise two-sided estimates on $\Lambda^{\pm}_\eps(cA_0, B_0)$. If one of the sets is a disjoint finite union of parallelepipeds and the other satisfies a very weak boundary regularity condition then we have sharp uniform bounds on $\Lambda^{\pm}_\eps(cA_0, B_0)$ from above. This in particular allows us to find two-term asymptotics for ${\rm{Tr}}f(S_{A,B})$ for very rough functions $f$. For the general case, the bound we obtain is off the conjectural one by at most a single logarithm. We will now present the cases we consider in the increasing order of generality.

\subsection{Case of two boxes}

In this section we consider the case $A = [0, a]^d$, $B = [0,b]^d$. Since the variables separate, the operator $S_{A,B}$ splits into a tensor power $S_{A,B} = \left(S_{[0,a],[0,b]}\right)^{\otimes d}$. In particular, the eigenvalues depend only on the product $c = ab$, and we get that the eigenvalues of $S_{A,B}$ are all the products of $d$ eigenvalues of the operator $S_{[0,a],[0,b]}$ counted with multiplicity. Using this we can generalize all of the estimates on the counting functions from dimension $1$ to higher dimensions with a simple combinatorial argument.

\begin{theorem}\label{two cubes}
Consider $A = B = [0, 1]^d$, $d\geq 1$, and $c\geq 2$. There exists $\alpha_d \geq 4$ such that
\begin{equation}\label{Lambda pm, upper, two cubes}
\Lambda_\eps^{\pm}(c A, B) \lesssim c^{d-1} \log(\tfrac{1}{\eps}) \log\left(\frac{\alpha_d \, c}{\log\left(\tfrac{1}{\eps}\right)}\right)
\end{equation}
uniformly for all $\alpha_d^{-c}< \eps < 1/2$. Moreover, if $\eps < c^{-\alpha_d}$ then we also have
\begin{equation}\label{Lambda pm, lower, two cubes}
\Lambda_\eps^{\pm}(c A, B) \gtrsim c^{d-1} \log(\tfrac{1}{\eps}) \log\left(\frac{\alpha_d \, c}{\log\left(\tfrac{1}{\eps}\right)}\right)
\end{equation}
If $\eps \leq \alpha_d^{-c}$ then there are no eigenvalues larger than $1-\eps$ and 
\begin{equation}\label{Lambda -, equivalence, two cubes}
\Lambda_\eps^- \asymp \left( \frac{\log(\tfrac{1}{\eps})}{\log \left(\tfrac{\log\left(\frac{1}{\eps}\right)}{c}\right)}\right)^d.
\end{equation}
\end{theorem}
The same tensor product structure can also be employed if $A$ and $B$ are axis-parallel boxes with not necessarily equal side lengths. In this way, by the same combinatorial argument, we get the following slightly more general result.
\begin{proposition}\label{two boxes}
Let $d\geq 1$ and $I_1,\dots, I_d$ and $J_1, \dots, J_d$ be fixed closed and finite intervals. Set $A = I_1 \times \dots \times I_d$ and $B = J_1 \times \dots \times J_d$. Then the conclusion of Theorem \ref{two cubes} holds, except $\alpha_d$ now also depends on the intervals $(I_k)$, $(J_k)$.
\end{proposition}

\subsection{General case}
For general sets $A,B$ the operator $S_{A,B}$ no longer splits into a tensor product, and the precision of our results will depend on the geometries of $A$ and $B$.  Specifically, we will either assume that they are disjoint finite unions of parallelepipeds, or that they are bounded and their boundaries have finite upper Minkowski content.
\begin{definition}
A set $\Gamma \subset \R^d$ is said to be of finite upper Minkowski content if
$$\limsup_{r\to 0^+} \frac{|\{x\in\R^d: {\rm{dist}} (x, \Gamma) < r\}|}{2r} < \infty.$$
\end{definition}
It is well known that for $\partial A$ regular enough, the limit $ \lim_{r\to 0^+} \frac{|\{x\in\R^d: {\rm{dist}} (x, \partial A) < r\}|}{2r}$ exactly agrees with the perimeter of $A$. Here "regular enough" includes any bounded set with Lipschitz (or $C^1$) boundary, see \cite[Theorem 2.106]{Ambrosio_00}. Note that since $2r$ goes to $0$ when $r\to 0^+$, any set of finite upper Minkowski content has zero Lebesgue measure. 

It is also not hard to see that, in general, if $\Gamma$ has finite upper Minkowski, then $\Gamma$ must must be bounded. For $d\ge 2$ the boundary of the set $A\subset\R^d$ is bounded if and only if either $A$ or $\R^d\backslash A$ is bounded; if $d = 1$ then either $A$ or $\R^d\backslash A$ is bounded, or each of them contains a ray. In either case, the only option for $A$ to have finite measure is if $A$ is bounded, so we will always assume that all of our sets are bounded.

\begin{definition}
We call a set $A$ a parallelepiped if there exists an invertible $d\times d$ matrix $M$ and a vector $b\in \R^d$ such that $A = M[0,1]^d + b$.
\end{definition}
\begin{theorem}\label{union and good}
Assume that $A$ is a bounded set whose boundary has finite upper Minkowski content and $B$ is a finite union of parallelepipeds with disjoint interiors such that both $A$ and $B$ have positive measures. There exists $\alpha =\alpha(d,A,B)\ge 4$ such that for all $c\ge 2$ and $\alpha^{-c} < \eps < \frac{1}{2}$ we have
\begin{equation}\label{eps not tiny}
\Lambda_\eps(cA, B) \lesssim c^{d-1} \log(\tfrac{1}{\eps}) \log\left(\frac{\alpha \, c}{\log\left(\tfrac{1}{\eps}\right)}\right).
\end{equation}
For $\eps \le \alpha^{-c}$ there are no eigenvalues larger than $1-\eps$ and
\begin{equation}\label{eps tiny}
\Lambda^{-}_\eps(cA, B)\asymp \left( \frac{\log(\tfrac{1}{\eps})}{\log \left(\tfrac{\log\left(\frac{1}{\eps}\right)}{c}\right)}\right)^d.
\end{equation}
\end{theorem}

\begin{remark}
If we restrict to finite unions of axis-parallel boxes then we do not have to assume that their interiors are disjoint, as we can do a finite subdivision to make them disjoint. 
\end{remark}
\begin{remark}
 We need to assume that $A$ and $B$ have positive measures only to ensure that we have a lower bound in \eqref{eps tiny}, since if $A$ or $B$ has zero measure then $S_{A,B} = 0$ and all of the eigenvalues are zero.   
\end{remark}

\begin{theorem}\label{good and good}
Assume that $A$ and $B$ are bounded sets whose boundaries have finite upper Minkowski content and such that both $A$ and $B$ have positive measures. There exists $\alpha =\alpha(d,A,B)\ge 4$ such that for all $c \ge 2$ and  $\alpha^{-c} < \eps < \frac{1}{2}$ we have
$$\Lambda_\eps(cA, B) \lesssim c^{d-1} \log(\tfrac{1}{\eps}) \log^2\left(\frac{\alpha \, c}{\log\left(\tfrac{1}{\eps}\right)}\right).$$
For $\eps \le \alpha^{-c}$ there are no eigenvalues larger than $1-\eps$ and
$$\Lambda^{-}_\eps(cA, B)\asymp \left( \frac{\log(\tfrac{1}{\eps})}{\log \left(\tfrac{\log\left(\frac{1}{\eps}\right)}{c}\right)}\right)^d.$$
\end{theorem}

When we pass from Theorem \ref{union and good} to Theorem \ref{good and good} we have to  put an additional square on the term $\log\left(\frac{\alpha \, c}{\log\left(\tfrac{1}{\eps}\right)}\right)$. While we believe that this effect does not actually take place, we are currently unable to remove it.
\begin{remark}
In both of these results we are also able to vary $A$ and $B$ with $c$ as long as they remain uniformly bounded, contain uniformly bounded from below balls inside of them and we have uniform bounds in the definition of finite upper Minkowski content for their boundaries for all $0 < r < 1$, say. We leave these routine generalizations to the interested reader.
\end{remark}
Arguments similar to our proof of Theorem \ref{two cubes} have already appeared in the literature, see e.g. \cite{bad_combinatorics} and \cite{Israel1}. However, both of them relied on estimates of Karnik, Romberg and Davenport \cite{Karnik_Romberg_Davenport}, which are not the strongest available bounds in all regimes in dimension $1$, see Section 2 below. In particular, in the regime $\eps < \alpha_d^c$, while we obtain at worst 
$$\Lambda_\eps^-(cA, B) \lesssim \left(\log\frac{1}{\eps}\right)^d,$$
their upper bound is $\left(\log c \log\frac{1}{\eps}\right)^d$, that is, they are losing a factor of at least $(\log c)^d$. 

For the general case, our assumption of boundaries having finite upper Minkowski content is weaker than all of the assumptions previously appearing in the literature, particularly the maximally Ahlfors regular boundary assumption in \cite{Marceca, Israel2}, see \cite[Theorem 2.104]{Ambrosio_00}. Moreover, our estimates in all regimes are at least one logarithm better than the bounds obtained in these papers, in particular being sharp in the setting of Theorem \ref{union and good}. In the proof strategy below we will highlight the key new estimates which allowed us to get this improvement.

The previous best estimate in the case of Theorem \ref{union and good} is due to Israel and  Mayeli \cite{Israel1}. They used wave packet basis to get a bound of the form
$$\Lambda_\eps(cA, B) \lesssim_\delta c^{d-1}(\log c)^{2+\delta}$$
in the regime $\eps \sim \frac{1}{c^s}$ with fixed $s > 0$ for any $\delta > 0$, thus being only $(\log c)^\delta$ away from our result. In the proof they relied on Gevrey classes, that is on the possible decay of the Fourier transform of the compactly supported functions. Thus, due to the Beurling--Malliavin theorem \cite{Beurling}, it is not possible to use their approach to match our result even in this regime. The same applies to the recent work \cite{new_one_right} compared to our Theorem \ref{good and good}.

\subsection{Proof strategy}
The first step in all of our arguments is the $S-S^2$ trick. We have $\eps < \lambda < 1-\eps$ if and only if $\eps(1-\eps)< \lambda(1-\lambda)$. For the time-frequency localization operator $S_{A,B}$ we have
$$S_{A,B}-S_{A,B}^2 = P_A Q_BQ_B P_A - P_A Q_B P_A Q_B P_A =P_A Q_B (1-P_{A}) Q_B P_A = P_AQ_BP_{A^c}Q_BP_A,$$
where $A^c$ is the complement of $A$. This latter expression is equal to $T^*T$ where $T = P_{A^c}Q_BP_A$. Therefore, we have to bound from above the number of singular values of $T$ larger than $\sqrt{\eps(1-\eps)}$.

Next, we want to partition the sets $A$ and $B$ into smaller sets and write $T$ as a sum of operators
$$T = \sum_{k, l} P_{A^c}Q_{B_l}P_{A_k}.$$

In our initial argument we applied Weyl's inequality to this sum, following the approach in \cite{Karnik_Romberg_Davenport}. However, there the authors decomposed $T$ as a sum of only $3$ operators, so it was enough for them to divide $\eps$ by $3$. In our setting the number of terms is of order $c^d$, so this approach gave sharp results only for $\eps < \frac{1}{c^d}$. Instead of applying Weyl's inequality, we will use the p-Schatten quasi-norm approach from \cite{Israel2}.

Let $K:H\to H$ be a compact operator from a Hilbert space $H$ to itself. Let $\sigma_1(K) \ge \sigma_2(K) \ge \ldots \ge 0$ be the sequence of its singular values. For $0 < p < \infty$ we define its $p$-Schatten quasi-norm by
$$\|K\|_p^p = \sum_{k=1}^\infty \sigma_k(K)^p.$$
If $K$ is not compact then we define $\|K\|_p = \infty$.

For $p = \frac{1}{2\log \frac{1}{\delta}}$ with $0<\delta<\frac{1}{2}$ we have 
\begin{equation}\label{schatten to count}
\|K\|_p^p \ge \sum_{k: \sigma_k(K)\ge \delta} \sigma_k(K)^p \ge \frac{1}{\sqrt{e}}|\{k: \sigma_k(K) \ge \delta\}|,
\end{equation}
where in the second step we used that $\delta^p = \frac{1}{\sqrt{e}}$. Thus, the number of singular values larger than $\delta$ is at most $\sqrt{e}\|K\|_p^p$. Note that for all $0 < \delta < \frac{1}{2}$ we have $0 < p < 1$. The key non-trivial fact \cite[Theorem 11.5.9]{Birman_Solomyak} about the $p$-Schatten quasi-norms for $0 < p < 1$ is that they satisfy the following subadditivity property
\begin{equation}\label{quasi norm subadditive}
\|K_1+K_2\|_p^p \le \|K_1\|_p^p + \|K_2\|_p^p.   
\end{equation}
Thus, it suffices to estimate the $p$-Schatten quasi-norms of the operators $P_{A^c}Q_{B_l}P_{A_k}$ with $p = \frac{1}{\log(\frac{1}{\eps(1-\eps)})}$.

For clarity we will describe our decomposition only when $B = [0, 1]^d$ and $A$ is a set with boundary of finite upper Minkowski content. In this case we do not partition $B$. The partition $(A_k)$ of $A$ will arise from the Whitney decomposition of the interior of $A$ into dyadic cubes. We keep all the cubes in the Whitney decomposition with side lengths larger than some threshold $2^{-D}$, and cover the remaining part by sets contained in cubes with side length $2^{-D}$. This gives $(A_k)$. From the assumption that $\partial A$ has finite upper Minkowski content we can estimate the number of cubes with side length $2^t$ by $C2^{-(d-1)t}$ for a constant $C$.

For a single term of the form $P_{U}Q_{V}P_{W}$ it is easy to see that if we enlarge $U$ or $W$ then the singular values can only increase.  For the cube $A_k$ from our partition we will enlarge $A^c$ to the complement of $(2A_k)$ which contains $A^c$ by the properties of the Whitney decomposition. In the case of cubes at the threshold we enlarge $A^c$ to the whole $\R^d$ and also enlarge $A_k$ to the corresponding cube with side length $2^{-D}$ containing it.

Splitting $(2A_k)^c$ additionally into $2^d-1$ regions, we can reduce the analysis of $P_{(2A_k)^c} Q_B P_{A_k}$ and $Q_B P_{A_k}$  to the study of tensor products of the one-dimensional operators 
$$I_r=Q_{[0,1]}P_{[0,r]}\quad \text{and}\quad J_r=P_{(-\infty, -2r]\cup[2r, +\infty)}Q_{[0,1]}P_{[-r,r]}.$$

Note that $I_r$ is just the usual time-frequency localization operator, while $J_r$ comes from the separation in the Whitney decomposition. For these operators we show the following estimates on the singular values.
\begin{lemma}\label{One dimensional operators}
There exist $\tau > 0$, $C > 0$, and $r_0 > 0$ such that for all $r > r_0$ and all $n\in\N$ we have
$$\sigma_n(I_r) \le\begin{cases}1,&  n < 10r,\\
Ce^{-\tau n}, & n \ge 10r,
\end{cases}$$
$$\sigma_n(J_r) \le Ce^{-\tau n}.$$
\end{lemma}
The estimate for $\sigma_n(I_r)$ follows from Theorem \ref{Bonami2} stated below (and is also explicitly or implicitly contained in many previous publications, see e.g. \cite{Osipov,Bonami_kk}), but we will present a self-contained complex-analytic proof inspired by \cite{kulikov_2026}.

The proof of the estimate for $\sigma_n(J_r)$ uses the precise structure of the kernel of $J_r$. Specifically, it is possible to decompose $J_r = \sum_{n=0}^\infty J_{r,n}$ into a sum of rank 2 operators $J_{r,n}$ with $\norm{J_{r,n}} \lesssim 2^{-n}$, from which the result follows. Let us remark that our initial argument was more involved and proceeded through the low displacement-rank structure and Zolotorev numbers estimates employed by Karnik, Romberg and Davenport in \cite{Karnik_Romberg_Davenport}.

The fact that the estimate for $\sigma_n(J_r)$ does not depend on $r$ is the key new ingredient in our proof compared to the argument in \cite{Israel2}. It ultimately stems from our use of the separation condition in the Whitney decomposition which was not used in the previous works.

Doing the combinatorial argument similar to the proof of Theorem \ref{two cubes}, combining all of our summands by means of \eqref{quasi norm subadditive} and carefully choosing the threshold $D$ for the Whitney decomposition we get the desired estimate for $\eps > \alpha^{-c}$. For $\eps \le \alpha^{-c}$ we can simply compare with the case of two boxes, as this is exactly the regime when the boundary term becomes proportional to the volume term.

\subsection{Area laws} We are interested in the traces ${\rm Tr} f(S_{A,B})$ for general functions $f:[0,1] \to \Cm$. For a compact self-adjoint operator $T:H\to H$ with eigenvalues $1 \ge\lambda_1\ge \lambda_2 \ge \ldots \ge 0$ and corresponding normalized eigenvectors $v_1, v_2, \ldots$ we define $f(T)(v) := \sum_{n=1}^\infty f(\lambda_n) \langle v, v_n\rangle v_n$. This operator is trace class whenever ${\rm Tr} f(T) = \sum_{n=1}^\infty f(\lambda_n)$  converges absolutely. We will always assume that $f(0) = 0$ for convenience since all operators that we consider have infinite-dimensional kernels. First, we establish results guaranteeing that $f(T)$ is trace class. 

For a function $f:[0, 1]\to \Cm$ we let $M_0 f(t) = \sup_{0 \le x \le t}|f(x)|$ and $M_1 f(t) = \sup_{1-t \le x \le 1}|f(1)-f(x)|$.
\begin{definition}
Consider a function $f:[0,1] \to \Cm$. We call $f$ trace class admissible for $L^2(\R^d)$ if there is $\delta>0$ such that
\begin{equation}\label{trace class condition}
\int_0^{\delta} \frac{M_0 f(\eps) \log\left(\frac{1}{\eps}\right)^{d-1}}{\eps\left(\log\log\left(\frac{1}{\eps}\right)\right)^d} \, d\eps < \infty. 
\end{equation}
\end{definition}
Note that for any trace class admissible for $L^2(\R^d)$ function $f$ we must necessarily have $\lim_{x\to 0^+} f(x) = f(0)=0$.

\begin{theorem}\label{trace class sufficient}
If $A, B\subseteq \R^d$ are bounded sets and $f$ is trace class admissible for $L^2(\R^d)$, then $f(S_{A,B})$ is trace class.
\end{theorem}

\begin{theorem}\label{trace class necessary}
Let $A, B\subseteq \R^d$ be such that $S_{A,B}$ is compact and such that their interiors are non-empty. If $f:[0,1]\to\Cm$ is such that $|f(x)|$ is non-decreasing near $0$ and $f(S_{A,B})$ is trace class then $f$ is trace class admissible for $L^2(\R^d)$.
\end{theorem}

We now turn to our main application of the uniform bounds from Theorem \ref{two cubes} and Theorem \ref{union and good}, namely to compute a two-term asymptotic formula for $\operatorname{Tr} f(S_{cA,B})$ in the limit $c\to \infty$ for extremely general functions $f$. As an input we use the corresponding formula for polynomials $f$. For $d=1$ this is supplied by the work of Landau and Widom \cite{widom_landau}. The higher dimensional case is considerably more difficult. The formula was conjectured by Widom \cite{Widom_conj} and much later established by Sobolev \cite{Sobolev13} for regular enough sets $A$ and $B$. For the context of the present paper, we rely on the sharp generalization of Sobolev's result from the recent work \cite{EE_paper} by Fournais, Seiringer, Solovej, and the second author. It was established that, for polynomials $f$ and sets $A$, $B$ with finite measure and finite perimeter,
\begin{equation}\label{polynomial case}
\begin{aligned}
\MoveEqLeft \operatorname{Tr} f(S_{cA,B})\\
&= c^{d} |A| |B|f(1) +  c^{d-1}\log(c)I( A, B)\int_{0}^1 \frac{f(\theta) - f(1)\theta}{\theta(1-\theta)} \, d\theta + o(c^{d-1}\log(c)),
\end{aligned}
\end{equation}
as $c\to\infty$, for a certain boundary coefficient $I(A,B)$. If $A$ and $B$ are sets with $C^1$ boundaries, then 
\[
I(A,B) = \frac{1}{4\pi^2} \int_{\partial A} \int_{\partial B} |\nu_A(x) \cdot \nu_B(p)| \, d\mathcal{H}^{d-1}(p) \, d\mathcal{H}^{d-1}(x),
\]
where $\nu_A$ and $\nu_B$ denote the interior normals to $A$ and $B$, respectively, and $d\mathcal{H}^{d-1}$ is the $d-1$-dimensional Hausdorff measure. Tools from geometric measure theory are necessary to describe the coefficient in the finite perimeter setting. We refer to \cite{EE_paper} for the details.  

For our work it is important to note that if the boundary of a set has finite upper Minkowksi content then this set has finite perimeter, thus \eqref{polynomial case} applies in our case. This is so because any set of finite upper Minkowski content has finite $d-1$-dimensional Hausdorff measure (see e.g. Lemma \ref{shell size boundary} below) and the fact that if the boundary of a set has finite $d-1$-dimensional Hausdorff measure then it has finite perimeter \cite[Proposition 3.62]{Ambrosio_00}. Note that even in this case the factor $I(A,B)$ can only be defined with the help of geometric measure theory.

Since Sobolev's original result, significant effort has been put to extend the two-term asymptotic formula \eqref{polynomial case} from polynomials to rougher spectral functions \cite{Sobolev_quasi_est,Sobolev_quasi_abstract}. This is in part motivated from physics. Here it was realized by Gioev and Klich \cite{Gioev_Klich} that the trace ${\rm Tr}f(S_{cA,B})$, with $f(\theta) = -\theta \log(\theta) - (1-\theta)\log(1-\theta)$ or related Rényi entropy functions, appears directly in the expression for the bipartite entanglement entropy for free Fermionic systems in the ground state. In this context, the domain $B$ represents the Fermi sea of the Fermionic operator and $A$ represents the spatial subsystem. Since the entropy functions satisfy $f(1) = 0$, the formula \eqref{polynomial case} expresses that $\operatorname{Tr} f(S_{cA,B}) \sim c^{d-1}\log(c)$, which is referred to as an enhanced area law in the physics literature. See \cite{Eisert_review} for relevant background. 

Previously, Sobolev carried out the extension argument relying on uniform Schatten quasi-norm estimates of the involved operators \cite{Sobolev_quasi_abstract}, which allowed $f$ with Hölder-type singularities. With the much more precise spectral bounds from Theorem \ref{two cubes} and Theorem \ref{union and good}, we are able to push this much further, in particular obtaining essentially necessary and sufficient conditions on the function $f$.

\begin{definition}
Consider a function $f:[0,1] \to \Cm$. We call $f$ area law admissible if
\begin{equation}\label{area law condition}
\int_{0}^1 \frac{M_0 f(\eps) + M_1 f(\eps)}{\eps} \, d\eps < \infty.
\end{equation}
\end{definition}

\begin{theorem}\label{good area law}
Let $A\subseteq \R^d$ be a bounded set whose boundary has finite upper Minkowski content and let $B$ be a finite union of parallelepipeds with disjoint interiors. Assume that $f$ is both  area law admissible and trace class admissible for $L^2(\R^d)$. There exist $c_0(f, A, B)$ and $C(A, B)$ such that for $c > c_0(f, A, B)$ we have
\begin{equation}\label{area law uniform bound}
|\operatorname{Tr} f(S_{cA,B})- c^d |A||B| f(1)|\leq C(A,B) c^{d-1}\log(c) \int_0^1 \frac{M_0 f(\eps) + M_1 f(\eps)}{\eps} \, d\eps.
\end{equation}

If in addition $f$ is Riemann integrable on $[\eps, 1-\eps]$ for all $0 < \eps < \frac{1}{2}$ then we also have
\begin{equation}\label{area law two term}
\begin{aligned}
\MoveEqLeft \operatorname{Tr} f(S_{cA,B})\\
&= c^{d} |A| |B|f(1) +  c^{d-1}\log(c)I( A, B)\int_{0}^1 \frac{f(\theta) - f(1)\theta}{\theta(1-\theta)} \, d\theta + o(c^{d-1}\log(c)),
\end{aligned}
\end{equation}
as $c\to \infty$.
\end{theorem}
\begin{remark}
It is crucial for us that $C(A,B)$ does not depend on $f$. On the other hand, $c_0(f, A, B)$ can and will depend on $f$, specifically on the integral in the definition of trace class admissibility for $L^2(\R^d)$.
\end{remark}


For $d = 1$ every $f$ which is area law admissible is automatically trace class admissible for $L^2(\R^d)$. On the other hand, for every $d\ge 2$ there exists a function which is area law admissible but which is not trace class admissible for $L^2(\R^d)$ and which is monotone near $0$. Thus, by Theorem \ref{trace class necessary} for any sets $A, B\subset \R^d$ with non-empty interiors we have ${\rm Tr} f(S_{A,B}) = \infty$ even though $f$ is area law admissible.

\begin{example}
Let $f(\theta) = \frac{1}{\log(\frac{2}{\theta})^{3/2}}$. Then $f$ is monotone increasing on $[0, 1]$, it is area law admissible but it is not trace class admissible for $L^2(\R^d)$ for any $d\ge 2$.
\end{example}
For general sets $A, B$ our argument gives an upper bound with one extra logarithm.
\begin{theorem}\label{bad area law}
Let $A, B\subseteq \R^d$ be bounded sets whose boundaries have finite upper Minkowski content. Assume that $f$ is both  area law admissible and trace class admissible for $L^2(\R^d)$. There exist $c_0(f, A, B)$ and $C(A, B)$ such that for $c > c_0(f, A, B)$ we have
\begin{equation*}
|\operatorname{Tr} f(S_{cA,B})- c^d |A||B| f(1)|\leq C(A,B) c^{d-1}\log(c)^2 \int_0^1 \frac{M_0 f(\eps) + M_1 f(\eps)}{\eps} \, d\eps.
\end{equation*}
\end{theorem}

\subsection{Lower order terms for $\operatorname{Tr} S^2$}
If both $A$ and $B$ are finite unions of axis-parallel boxes then we are able to go further in the expansion in the simplest non-trivial case $f(\theta) = \theta^2$. We carry out this computation in part out of mathematical curiosity, 
but also because there seems to be some interest from a physics point of view for further terms in the expansion of $\operatorname{Tr}f(S_{cA,B})$, specifically for $f$ being the entropy function. For instance, Kitaev and Preskill argued in \cite{Kitaev_preskill} from a physical understanding of the entanglement entropy that the constant order term for $d=2$ in the expansion of $\operatorname{Tr} f(S_{cA,B})$ should carry topological information. While it is not our intention to place their claims on a rigorous mathematical footing, nor to verify them, we still see it as an interesting mathematical problem to obtain further terms for general functions $f$.

\begin{theorem}\label{TrS2}
Let $A, B\subseteq \R^d$ be finite unions of axis-parallel boxes. The trace of $S^2_{A,B}$ can be computed in terms of the side lengths of boxes constituting $A$ and $B$ by means of four standard arithmetic operations, exponentiation, taking logarithms and the exponential integral function ${\rm E_1}(z) = \int_z^\infty \frac{e^{-t}}{t}$.
\end{theorem}
Unfortunately, if $A$ and $B$ consist of respectively $n$ and $m$ boxes then our procedure gives us of order $9^d n^2m^2$ terms, making it fairly infeasible to do in practice. So, we will do an explicit computation only in the case $A = [0, c]$, $B = [0, 1]$.
\begin{theorem}\label{TrS2 explicit}
The following expansion holds:
\begin{align*}
\MoveEqLeft \operatorname{Tr}S_{[0,c], [0,1]}^2\\
&= c - \frac{\log(c)}{\pi^2} - \frac{1+\gamma + \log(2\pi)}{\pi^2} + \left(\frac{2}{\pi} c\left({\rm Si}(2\pi c)-\frac{\pi}{2}\right) + \frac{\cos(2\pi c)}{\pi^2} + \frac{{\rm Ci}(2\pi c)}{\pi^2}\right)\\
&=c - \frac{\log(c)}{\pi^2} - \frac{1+\gamma + \log(2\pi)}{\pi^2}\\
&\quad - \frac{\cos(2\pi c)}{\pi^2}\sum_{n=1}^\infty (-1)^{n} \frac{(2n)!-(2n-1)!}{(2 \pi c)^{2n}} - \frac{\sin(2\pi c)}{\pi^2} \sum_{n=1}^\infty (-1)^n \frac{(2n+1)!-(2n)!}{(2\pi c)^{2n+1}},
\end{align*}
where ${\rm Si(t)} = \int_0^t \frac{\sin (x)}{x}dx$ is the sine integral, ${\rm Ci(t)} = -\int_t^\infty \frac{\cos(x)}{x}dx$ is the cosine integral and $\gamma$ is the Euler--Mascheroni constant. The series is asymptotic, meaning that if we take the first $N$ terms in both sums then the error will be $O\left(\frac{1}{c^{2N+2}}\right)$.
\end{theorem}
Note that the oscillations in $\operatorname{Tr}S_{[0,c], [0,1]}^2$ start appearing only at $O(\frac{1}{c^2})$ term. Our argument shows that they can potentially appear already in $O(1)$ term, but it turned out that both $O(1)$ and $O(\frac{1}{c})$ oscillating terms cancelled out. We do not have an a priori explanation for this phenomenon.

\subsection{One-term asymptotics}
If we are only interested in the first term in the expansion \eqref{area law two term} then the conditions on $A, B$ and $f$ can be greatly relaxed, in particular we do not need to assume anything about the geometries of $\partial A$ and $\partial B$.

Since for any $A, B$ of finite measure ${\rm Tr}S_{A,B} = |A| |B|$ is finite, for any function $f:[0,1]\to\Cm$ such that $|f(\theta)|\le C\theta, \theta\in[0, 1]$ we have that $f(S_{A,B})$ is trace class. If in addition we assume that $f$ is continuous at $1$ then we can establish a one-term asymptotic formula for ${\rm Tr}f(S_{cA,B})$.

\begin{theorem}\label{one term general}
Let $A, B \subseteq \R^d$ be sets with finite measure. If $f:[0,1]\to \Cm$ is continuous at $1$ and satisfies $|f(\theta)| \leq C \theta$ for some $C>0$, then
\begin{equation}\label{one-term}
\operatorname{Tr} f(S_{cA,B}) = c^d|A||B| f(1) + o(c^{d}),
\end{equation}
as $c\to \infty$.
\end{theorem}
Although the condition $|f(\theta)|\leq C\theta$ might seem restrictive, in this generality it is actually the optimal one.
\begin{proposition}\label{one term counterexample}
Given a function $f$ such that $\lim_{x\to 0^+} \frac{|f(x)|}{x} = \infty$ and a set $B\subseteq \R$ with finite and positive measure, there exists a set $A\subseteq \R$ of finite measure such that $f(S_{A,B})$ is not trace class. 
\end{proposition}
In particular, this applies to the entropy function $f(\theta) = -\theta \log(\theta) - (1-\theta) \log(1-\theta)$.

On the other hand, if the sets $A$ and $B$ are bounded then we can replace the assumption $|f(\theta)|\le C\theta$ with the assumption of $f$ being bounded and  trace class admissible for $L^2(\R^d)$ as in Theorem \ref{good area law}.

\begin{theorem}\label{one term bounded}
Let $A, B \subseteq \R^d$ be bounded sets and assume that $f$ is bounded, trace class admissible for $L^2(\R^d)$ and that $f$ is continuous at $1$. Then 
\[
\operatorname{Tr} f(S_{cA, B}) = c^d|A||B| f(1) + o(c^{d}),
\]
as $c\to \infty$.
\end{theorem}

\section{Previous results on one-dimensional operators}
In this section we will collect previous results on the distribution of eigenvalues of $S_{A,B}$ in the simplest case when $A$ and $B$ are one-dimensional intervals and restate them in the form that will be convenient for our use. By rescaling we can always assume that $A = [0, c], B = [0, 1]$ so that we have a sequence $(\lambda_n(c))$ of eigenvalues of $S_{A,B}$. The first result about their behaviour was found by Slepian \cite{Slepain_65} and rigorously proved by Landau and Widom \cite{widom_landau}.

\begin{theorem}\label{Slepian}
For all bounded intervals $A, B \subseteq \R$ and all fixed $0<a<1$ we have
\[
N_a(A,B) = |A||B| + \tfrac{1}{\pi^2}\log\big(\tfrac{1-a}{a}\big)\log(|A||B|) + o(\log (|A||B|)),
\]
as $|A||B| \to \infty$.
\end{theorem}
Note that this result is slightly different from the statement in \cite{widom_landau} as we use a different normalization of the Fourier transform. This also applies to some of the following results.

Their proof proceeded by first establishing two-term asymptotics for ${\rm Tr}\, S^n_{A,B}$ for all powers $n\in\N$ (even with $O(1)$ error term) and then approximating characteristic function of an interval by polynomials. However, since their argument had extremely poor uniformity in $n$ and thus extremely poor uniformity in $a$, it is virtually impossible to use their estimates to bound $N_a(A,B)$ and $\Lambda_a(A,B)$ for varying $a$. So, there was a lot of research on uniform estimates for these numbers when $a$ is varying \cite{Israel2015,Bonami_Karoui_17,Kulikov_24}. One of the best results, which is at the same time completely uniform, sharp and explicit, was recently obtained by Karnik, Romberg and Davenport \cite{Karnik_Romberg_Davenport}.

\begin{theorem}\label{Karnik}
For all $0<\eps<1/2$ and all bounded intervals $A$ and $B$ we have
\[
\Lambda_\eps(A,B) \leq \frac{2}{\pi^2} \log(50 |A||B|+25) \log\left(\frac{5}{\eps(1-\eps)}\right) + 7.
\]
\end{theorem}

Note that the dependence on $c$ and on $\eps(1-\eps)$ in this result is logarithmic just like in Theorem \ref{Slepian}, while being completely explicit in all constants. In particular, combining Theorem \ref{Karnik} with Theorem \ref{Slepian} is already enough to establish our version of the enhanced area law for the two-term asymptotics of ${\rm{Tr}} f(S_{A,B})$
$${\rm{Tr}} f(S_{A,B}) =f(1)|A||B|+ \frac{1}{\pi^2} \log(|A||B|) \int_0^1 \frac{f(\theta) - \theta f(1)}{\theta(1-\theta)} \, d\theta + o(\log(|A||B|))$$
for extremely general functions $f$, see Theorem \ref{good area law} for the precise statement.

This result is essentially sharp for a very wide range of values of $\eps$. However, for very small values of $\eps$ (almost exponentially small in $|A||B|$) it is possible to get a better bound. For the eigenvalues $\lambda_n(c)$ with $n < c$ this was done by the first author \cites{kulikov_2026}. 

\begin{theorem}\label{Nazarov}
There exist numbers $c_0,B, \eta, \mu > 0$ such that for $c > c_0$ and $n < c-B\log^2(c)$ we have
\begin{equation}\label{Nazarov equation}    
\exp\left(-\mu\frac{c-n}{\log(\frac{2c}{c-n})}\right) < 1-\lambda_n(c) <\exp\left(-\eta\frac{c-n}{\log(\frac{2c}{c-n})}\right).
\end{equation}
\end{theorem}

For $n > c$ not only is it possible to obtain better estimates, but even an asymptotic formula for $\lambda_n(c)$ with a tiny relative error was obtained by Bonami and Karoui \cite{Bonami_Karoui_17}.
\begin{theorem}
For all $n \ge c \ge 10$ we have
$$\lambda_n(c) = \exp\left(-\frac{\pi^2\left(n+\frac{1}{2}\right)}{2} \int_{\Phi\left(\frac{c}{n+\frac{1}{2}}\right)}^1 \frac{1}{tE(t)^2}dt+O(\log(n))\right),$$
where $E(t) = \int_0^1 \sqrt{\frac{1-t^2x^2}{1-x^2}}dx$ is the elliptic integral of the second kind and $\Phi$ is the inverse of the function $t\to \frac{t}{E(t)}$. 	
\end{theorem}

Using known asymptotics of the elliptic integral near $0$ and $1$ we can restate it in the following (much cruder) form which has an advantage of not involving any special functions.
\begin{theorem}\label{Bonami2}
There exist numbers $c_0,B, \eta, \mu > 0$ such that for $c > c_0$, $2c > n > c+B\log^2(c)$ we have
$$\exp\left(-\mu\frac{n-c}{\log(\frac{2c}{n-c})}\right) < \lambda_n(c) <\exp\left(-\eta\frac{n-c}{\log(\frac{2c}{n-c})}\right)$$
while for $n \ge 2c$ we have
$$\exp\left(-\mu n \log\left(\frac{n}{c}\right)\right) < \lambda_n(c) < \exp\left(-\eta n \log\left(\frac{n}{c}\right)\right).$$	
\end{theorem}

To actually apply these estimates, we want to reformulate them in terms of counting functions as well. Note that for the lower bound on $N_{1-\eps}([0,1],[0,c])$ and the upper bound on $N_{\eps}([0,1],[0,c])$ we can combine Theorem \ref{Nazarov} and Theorem \ref{Bonami2}, respectively, with Theorem \ref{Karnik} and Theorem \ref{Slepian} to get estimates valid for all $\eps < \frac{1}{2}$. For the remaining bounds we do not have an analogue of Theorem \ref{Karnik} at our disposal, so they will only work for $\eps < c^{-A}$ for a large enough constant $A$. Specifically, we have the following six assertions. Let $\kappa, c_1 > 0$ be large enough constants and $c_2 > 0$ be a small enough constant, independent of $c$ and $\eps$, and we assume that $c > c_1$. 

Whenever $2^{-\kappa c} < \eps < \frac{1}{2}$, we have
\begin{equation}\label{N 1-e lower}
N_{1-\eps}([0,c],[0,1]) \ge c - c_1 \log\left(\frac{1}{\eps}\right) \log\left(\frac{\kappa c}{\log(\frac{1}{\eps})}\right)
\end{equation}
and
\begin{equation}\label{N e upper}
N_\eps([0,c],[0,1]) \le c +c_1 \log\left(\frac{1}{\eps}\right) \log\left(\frac{\kappa c}{\log(\frac{1}{\eps})}\right).
\end{equation}

For $2^{-\kappa c} < \eps < c^{-A}$ we have complementary bounds as well
\begin{equation}\label{N 1-e upper}
N_{1-\eps}([0,c],[0,1]) \le c  - c_2 \log\left(\frac{1}{\eps}\right) \log\left(\frac{\kappa c}{\log(\frac{1}{\eps})}\right)
\end{equation}
and
\begin{equation}\label{N e lower}
N_\eps([0,c],[0,1]) \ge c + c_2 \log\left(\frac{1}{\eps}\right) \log\left(\frac{\kappa c}{\log(\frac{1}{\eps})}\right).
\end{equation}

For $\eps \le 2^{-\kappa c}$ we have
$$N_{1-\eps}([0,c],[0,1]) = 0$$
and
$$N_\eps([0, c],[0, 1]) \asymp \frac{\log\left(\frac{1}{\eps}\right)}{\log\left(\frac{4\log\left(\frac{1}{\eps}\right)}{c}\right)}.$$

Let us, as an illustration, deduce \eqref{N 1-e lower} from Theorems \ref{Slepian}, \ref{Karnik} and \ref{Nazarov}, and leave the other five cases to the interested reader. 

For $\frac{1}{100} \le \eps < \frac{1}{2}$ we have $N_{1-\eps}([0,c],[0,1])\ge N_{99/100}([0,c],[0,1]) = c + O(\log c)$ by Theorem \ref{Slepian}. Next, we notice that if $n_c = [c]$, where $[x]$ is the largest integer not greater than $x$, and $c$ is large enough then $\frac{1}{3} \le \lambda_{n_c}(c) \le \frac{2}{3}$. Indeed, by Theorem \ref{Slepian} we have $N_{1/3}([0,c],[0,1]) = c + \frac{\log(2)}{\pi^2}\log(c) + o(\log(c))$ and $N_{2/3}([0,c],[0,1]) = c - \frac{\log(2)}{\pi^2}\log(c) + o(\log(c))$, thus $n_c$ does not satisfy $\lambda_n(c) > \frac{2}{3}$ but does satisfy $\lambda_n(c) \ge \frac{1}{3}$. 

If $c^{-A} \le \eps < \frac{1}{100}$ then $n_c$ is in the set $\{n: \eps < \lambda_n(c) \le 1-\eps\}$. Note also that this set is clearly a segment of integers. We therefore have $$N_{1-\eps}([0,1],[0,c])\ge n_c - |\{n: \eps < \lambda_n(c) < 1-\eps\}| = n_c - \Lambda_\eps([0, 1], [0, c]).$$
Applying Theorem \ref{Karnik} and $n_c \ge c$ for large enough $c$ we get
$$N_{1-\eps}([0,1],[0,c])\ge c - 100\log(c) \log \left(\frac{1}{\eps}\right).$$
If $c^{-A} \le \eps$ and $c$ is large enough then $\log\left(\frac{\kappa c}{\log(\frac{1}{\eps})}\right) \ge \frac{1}{2}\log(c)$, so we have to at most double the constant. In fact, in \cite{kulikov_2026} in the proof of Theorem \ref{Nazarov} the same argument was used for a much wider range of $\eps$, up to $\eps = \exp(-c^{7/8})$.

If $0 < \eps < c^{-A}$ then first of all $c_1 \log\left(\frac{1}{\eps}\right) \log\left(\frac{\kappa c}{\log(\frac{1}{\eps})}\right) \ge B\log^2(c)$. Thus, it suffices to show that for $n = \left[c-c_1 \log\left(\frac{1}{\eps}\right) \log\left(\frac{\kappa c}{\log(\frac{1}{\eps})}\right)\right]+1$ the right-hand side in \eqref{Nazarov equation} is at most $\eps$ (if $n\leq 0$ then the required estimate holds automatically). For large enough $c_1$ we clearly have $n < c - \frac{c_1}{2}\log\left(\frac{1}{\eps}\right) \log\left(\frac{\kappa c}{\log(\frac{1}{\eps})}\right)$. So, it is enough to show that for large enough $c_1$ we have
$$\log\left(\frac{1}{\eps}\right) \le\eta\frac{c_1\log\left(\frac{1}{\eps}\right) \log\left(\frac{\kappa c}{\log(\frac{1}{\eps})}\right)}{2\log\left(\frac{2c}{c_1\log\left(\frac{1}{\eps}\right) \log\left(\frac{\kappa c}{\log(\frac{1}{\eps})}\right)}\right)}.$$
Dividing by $\log\left(\frac{1}{\eps}\right)$ and putting $t=\frac{\kappa c}{\log\left(\frac{1}{\eps}\right)}$ we get an inequality
$$1\le \frac{\eta c_1}{2} \frac{\log(t)}{\log(t)-\log\log(t)-\log(\frac{c_1}{2})},$$
which is always true if $c_1 > \frac{2}{\eta}+2$. For \eqref{N 1-e upper} we would instead need to show that similar expression with $\mu$ and $c_2$ is at most $1$, and for this we would use that $\log \log(t) \le \frac{\log(t)}{2}$, say.
\section{Case of two boxes}
In this section we will prove Theorem \ref{two cubes}. The proof of Proposition \ref{two boxes} is similar so we leave its proof to the interested reader.

We begin with the following simple observation that will nevertheless be sufficient for the proof:
\begin{equation}\label{N trick}
N_{a^{1/d}}([0, c], [0,1])^d\le N_a([0,c]^d,[0,1]^d)\le N_a([0,c],[0,1])^d.
\end{equation}
Indeed, if the product of $d$ numbers, each of which is between $0$ and $1$, is larger than $a$ then each of them is also larger than $a$, which gives us an upper bound. On the other hand, if each of the numbers is larger than $a^{1/d}$ then their product is larger than $a$.

So, we can apply the estimates for $N_a([0,c],[0,1])$ from the previous section. Note that for $2^{-\kappa c} < \eps < \frac{1}{2}$ we always have  $\log\left(\frac{1}{\eps}\right) \log\left(\frac{\kappa c}{\log(\frac{1}{\eps})}\right) = O(c)$. When raising to the $d$'th power, we will use the following simple inequality: if $y = O(x)$ then $(x+y)^d = x^d + O(x^{d-1}y)$. We will only use it with $x = c$.

We begin with estimating $N_{1/2}([0,c]^d,[0,1])$ from \eqref{N trick}. We get
$$N_{2^{-1/d}}([0, c], [0,1])^d\le N_{1/2}([0,c]^d,[0,1]^d)\le N_{1/2}([0,c],[0,1])^d.$$

By Theorem \ref{Slepian} both $N_{1/2}([0,c],[0,1])$ and $N_{2^{-1/d}}([0,c],[0,1])$ are  $c + O(\log c)$. Raising this to the power $d$ we get
\begin{equation}\label{N 1/2}
N_{1/2}([0,c]^d,[0,1]^d) = c^d +O(c^{d-1}\log c).
\end{equation}

First, we are going to prove \eqref{Lambda pm, upper, two cubes}. For this we will use a lower bound on $N_{1-\eps}([0,c],[0,1])$ and an upper bound on $N_\eps([0,c],[0,1])$, respectively. For $\Lambda_\eps^+([0,c]^d,[0,1]^d)$ we get 
\begin{align*}
\Lambda_\eps^+([0,c]^d,[0,1]^d) &= N_{1/2}([0,c]^d,[0,1]^d) - N_{1-\eps}([0,c]^d,[0,1]^d)\\
&\le N_{1/2}([0,c]^d,[0,1]^d) - N_{(1-\eps)^{1/d}}([0,c],[0,1])^d.
\end{align*}
By Bernoulli's inequality we know that $(1-\eps)^{1/d} \le 1-\frac{\eps}{d}$ and therefore we have $N_{(1-\eps)^{1/d}}([0,c],[0,1]) \geq N_{1-\eps/d}([0,c],[0,1])$. Plugging in our bound we get for $2^{-\kappa c} < \eps < \frac{1}{2}$
$$\Lambda_\eps^+([0,c]^d,[0,1]^d) \le c^d + O(c^{d-1}\log c) - c^d + O\left(c^{d-1}\log\left(\frac{d}{\eps}\right) \log\left(\frac{\kappa c}{\log(\frac{d}{\eps})}\right)\right).$$
Canceling $c^d$, noting that the first big-O is dominated by the second and the fact that for $\eps < \frac{1}{2}$ we have $\log\left(\frac{d}{\eps}\right)\asymp \log\left(\frac{1}{\eps}\right)$ we get the desired upper bound after possibly increasing $\kappa$ to $\alpha_d$.

For $\Lambda_\eps^-([0,c]^d,[0,1]^d)$ we get by a similar reasoning for $2^{-\kappa c} < \eps < \frac{1}{2}$
$$\Lambda_\eps^-([0,c]^d,[0,1]^d) \le c^d + O\left(c^{d-1}\log\left(\frac{1}{\eps}\right)\log\left(\frac{\kappa c}{\log\left(\frac{1}{\eps}\right)}\right)\right)-c^d + O(c^{d-1}\log c).$$

Again, canceling $c^d$ and noting that the first big-O dominates the second one we get the desired result.

To prove \eqref{Lambda pm, lower, two cubes} we will instead employ an upper bound on $N_{1-\eps}([0,c],[0,1])$ and a lower bound on $N_\eps([0,c],[0,1])$. For this, just an inequality $(x+y)^d = x^d + O(x^{d-1}y)$ will not be enough, we will need that if $0\le y \le \frac{1}{2}x$ then $(x-y)^d \le x^d - c_d x^{d-1}y$ and $(x+y)^d \ge x^d + c_d x^{d-1}y$. Note that we can achieve $y \le \frac{1}{2}x$ by making {the constant $c_2$ in \eqref{N 1-e upper} and \eqref{N e lower}} smaller if necessary. The rest of the argument is essentially the same as in the proof of \eqref{Lambda pm, upper, two cubes}, the only changes are that we have to write inequalities explicitly instead of big-O and that in the case of $\Lambda^-_\eps([0,c]^d,[0,1])$ we have $\log\left(\frac{1}{\eps^{1/d}}\right)$ instead of $\log\left(\frac{1}{\eps}\right)$, but since they are proportional to each other we get the same result in the end.

It remains to cover the case $\eps \le 2^{-\kappa c}$. First of all, there are no eigenvalues larger than $1-\eps$ because the right-hand side of \eqref{N trick} is zero. For $\Lambda^-_\eps([0,c]^d,[0,1]^d)$ we get
$$\Lambda^-_\eps([0,c]^d,[0,1]^d) =N_\eps([0,c]^d,[0,1]^d)- N_{1/2}([0,c]^d,[0,1]^d).$$
For the upper bound we simply discard the second term and bound 
$$\Lambda^-_\eps([0,c]^d,[0,1]^d) \lesssim \left(\frac{\log\left(\frac{1}{\eps}\right)}{\log\left(\frac{4\log\left(\frac{1}{\eps}\right)}{c}\right)}\right)^d.$$
For the lower bound, first we pick some $\gamma_d > 0$ and if $\eps > \gamma_d^{-c}$ then we can use a lower bound for $\eps = 2^{-\kappa c}$ from the previous case (increasing $\kappa$ if needed) which is proportional to the lower bound in the current case with the obvious monotonicity of $\Lambda_\eps^-(A,B)$ in $\eps$.

Finally, in the case $\eps \le \gamma_d^{-c}$ we first crudely write 
$$\Lambda^-_\eps([0,c]^d,[0,1]^d) =N_\eps([0,c]^d,[0,1]^d)- N_{1/2}([0,c]^d,[0,1]^d) \ge N_\eps([0,c]^d,[0,1]^d)-2c^d$$
and estimate the first term with \eqref{N trick} as
$$N_\eps([0,c]^d,[0,1]^d) \ge \left( \frac{\delta\log\left(\frac{1}{\eps^{1/d}}\right)}{\log\left(\frac{4\log\left(\frac{1}{\eps^{1/d}}\right)}{c}\right)}\right)^d$$
for some $\delta > 0$. If $\eps \le \gamma_d^{-c}$ for small enough $\gamma_d$ then this is proportional to the desired lower bound and is also at least $4 c^d$ so we can cancel $-2c^d$ that we had.

\section{General case}
In this section we will prove Theorem \ref{union and good} and Theorem \ref{good and good}. We begin with the first one as it still contains most of our techniques while being slightly easier. So, let $A\subseteq \R^d$ be a bounded set with positive measure and boundary of finite upper Minkowski content and $B = \cup_n B_n$ be a finite union of parallelepipeds with disjoint interiors. We will also only consider the case $\eps > \alpha^{-c}$ for now and cover the other case at the end of the section.

We start with the $S-S^2$ trick. We have
$$S_{cA,B}-S_{cA,B}^2 = P_{cA}Q_BP_{c A^c}Q_BP_{cA} = T^* T,$$
where $T =P_{cA^c}Q_BP_{cA}$. We are interested in the number of eigenvalues of $S_{cA,B}-S_{cA,B}^2$ larger than $\eps(1-\eps)$, that is the number of singular values of $T$ larger than $\sqrt{\eps(1-\eps)}$. 

Since $\partial B_n$ has measure $0$, we clearly have $Q_B = \sum_{n=1}^N Q_{B_n}$, and therefore
$$P_{cA^c}Q_BP_{cA} = \sum_{n=1}^N P_{cA^c}Q_{B_n}P_{c A}.$$
Let $p = \frac{1}{\log \left(\frac{1}{\eps(1-\eps)}\right)} < 1$. By \eqref{schatten to count} and \eqref{quasi norm subadditive} we have
\begin{equation}\label{final bound Lambda}
\Lambda_\eps(cA, B)\le\sqrt{e}\|P_{cA^c}Q_BP_{cA}\|_p^p \le \sqrt{e}\sum_{n=1}^N \|P_{cA^c}Q_{B_n}P_{c A}\|^p_p.
\end{equation}
Using the $S-S^2$ trick in reverse we can rewrite these quasi-norms in terms of the eigenvalues $\lambda_m(cA, B_n)$ as 
$$\|P_{cA^c}Q_{B_n}P_{c A}\|^p_p = \sum_m (\lambda_m(cA, B_n)(1-\lambda_m(cA, B_n)))^{\frac{p}{2}}.$$
We have $B_n = M_n [0,1]^d + b_n$ for some invertible matrix $M_n$ and $b_n\in \R^d$. To reduce to the case $[0,1]^d$ we will use the following lemma.

\begin{lemma}\label{change of variables}
Let $A, B\subseteq \R^d$ be sets of finite measure and $M:\R^d \to \R^d$ be an invertible linear map. Then, for all $u,v\in \R^d$ and all $n\in \N$, 
\[
\lambda_n(A,B) = \lambda_n(MA + v, M^{-T}B +u).
\]
\end{lemma}
\begin{proof}
We start with the case $u = v = 0$. Consider the unitary dilation operator $D_Mf(x) = |\operatorname{det}(M)|^{\frac{1}{2}} f(Mx)$ with inverse $D_{M}^{-1} = D_{M^{-1}}$. We are going to compute $D_{M}^{-1} S_{A,B} D_{M}$. To this end, it is sufficient to note that, for any invertible matrix $M$ and all measurable sets $X$,
\[
\mathcal{F} D_{M} =  D_{M^{-T}} \mathcal{F}, \quad \mathcal{F}^{-1}D_{M^{-T}} = D_M \mathcal{F}^{-1}, \quad P_X D_M = D_M P_{MX}
\]
It follows that 
\[
D_{M}^{-1} S_{A,B} D_{M} = P_{MA} Q_{M^{-T}B} P_{MA}
\]
and therefore $\lambda_n(A,B) = \lambda_n(MA, M^{-T} B)$ by unitary equivalence. Next, we handle general translations $u,v$. Consider the unitary translation operator $T_vf(x) = f(x-v)$ and the phase shift multiplication operator $W_uf(p) = e^{-2\pi i u\cdot p} f(p)$ with inverses $T_{-v}$ and $W_{-u}$. As before, the following commutation relations hold:
\[
\mathcal{F}T_v = W_v \mathcal{F}, \quad \mathcal{F}^{-1} W_v = T_v \mathcal{F}^{-1}, \quad P_X T_v = T_v P_{X-v}, \quad P_X W_u = W_u P_X.
\]
We conclude that $W_u^{} T_v^{} S_{A,B} T_v^{-1} W_u^{-1} = S_{A+v, B+u}$ and therefore that $\lambda_n(A,B) = \lambda_n(A+v, B+u)$.

\end{proof}
Applying Lemma \ref{change of variables} to $A, B_n$ we get $\lambda_m(cA,B_n) = \lambda_m(cM_n^{-T}A, [0,1]^d)$. Observe that $M_n^{-T}A$ is still a bounded set whose boundary has finite upper Minkowski content. Indeed, simply note that for any $r>0$
\[
|M^{-T} A + B_r| \leq \frac{1}{|\operatorname{det}(M)|} \big\lvert A + B_{r\norm{M^{-T}}^{-1}} \big\rvert,
\]
where $+$ stands for the Minkowski sum of two sets. So, from now on to simplify the notation we will assume that $B = [0,1]^d$ and $A$ is some bounded set whose boundary has finite upper Minkowski content and we want to bound $\|P_{cA^c}Q_{[0,1]^d}P_{c A}\|^p_p$.

Next, we introduce the Whitney decomposition of $A$. To do this we  recall that $\partial A$ has measure $0$, hence $P_A = P_{{\rm int} A}$ and so we can assume that $A$ is open. Note also that $\partial ({\rm int} A)\subseteq\partial A$, hence $\partial ({\rm int} A)$ also has finite upper Minkowski content. The following is taken from \cite[Appendix J]{Grafakos_classical}

\begin{proposition}\label{usual Whitney}
Let $A\subsetneq \R^d$ be an open subset. There exists a family $(Q_k)$ of closed dyadic cubes such that
\begin{enumerate}
    \item $\bigcup_k Q_k = A$ and the $Q_k$'s have disjoint interiors,
    \item There are constants $c_1, c_2$ only dependent on $d$ such that
    \[
    c_1 \operatorname{diam}(Q_k) \leq \operatorname{dist}(Q_k, A^c) \leq c_2 \operatorname{diam}(Q_k).
    \]
\end{enumerate}
\end{proposition}
For the future use, we will need to modify this decomposition slightly to make constant $c_1$ as large as we like.
\begin{proposition}\label{our Whitney}
Let $A\subsetneq \R^d$ be an open subset. There exists a family $(Q_k)$ of closed dyadic cubes such that
\begin{enumerate}
    \item $\bigcup_k Q_k = A$ and the $Q_k$'s have disjoint interiors,
    \item There is a constant $c_3$ only dependent on $d$ such that
    \[
    2 \operatorname{diam}(Q_k) \leq \operatorname{dist}(Q_k, A^c) \leq c_3 \operatorname{diam}(Q_k).
    \]
\end{enumerate}
\end{proposition}
\begin{proof}
Let $Q_k$ be a sequence of cubes from Proposition \ref{usual Whitney} and cut each of them, for some $m\in \N$, into $2^{md}$ dyadic subcubes $Q_{k,1},\ldots , Q_{k, 2^{md}}$ of diameter $2^{-m}{\rm diam}(Q_k)$. We have 
$${\rm dist}(Q_{k,n}, A^c) \le {\rm dist}(Q_{k}, A^c) +{\rm diam}(Q_{k}) \le (c_2 +1){\rm diam}(Q_{k}) = (c_2+1)2^m {\rm diam}(Q_{k,n}),$$
so $c_3 = (c_2+1)2^m$ works. On the other hand, 
$${\rm dist}(Q_{k,n}, A^c) \ge {\rm dist}(Q_k, A^c) \ge c_1 {\rm diam}(Q_k)= c_1 2^{m}{\rm diam}(Q_{k,n}),$$
so if we choose $m$ so that $c_12^m > 2$ we will get the desired result for the cubes $Q_{k,n}$.
\end{proof}
By $A_l$ we denote the collection of cubes in the Whitney decomposition of $A$ with side length $2^{-l}$. Since we assume that $A$ is bounded, there is $l_0 \in \Z$ such that $A_l = \varnothing$ for  $l < l_0$. We will need the following two lemmas about open, bounded sets $A$ whose boundary has finite upper Minkowski content.

\begin{lemma}\label{shell size}
Let $A$ be an open bounded set with boundary of finite upper Minkowski content. There exists a constant $C$ depending only on $A$ such that
\[
|A_l| \leq C 2^{(d-1)l}.
\]
\end{lemma}
\begin{lemma}\label{shell size boundary}
Let $A$ be an open bounded set with boundary of finite upper Minkowski content. There exists a constant $C$ depending only on $A$ such that for all small enough $r > 0$ we can find a set $N_r$ of axis-parallel boxes with side length $r$ and with disjoint interiors such that $|N_r| \le Cr^{1-d}$ and 
$$\partial A + B_r \subseteq\cup_{Q\in N_r} Q.$$
\end{lemma}
\begin{proof}[Proof of Lemma \ref{shell size}]
Let $X_l = \cup_{Q\in A_l} Q$. We have
\[
|A_l| = \sum_{Q\in A_l} 1 = \sum_{Q\in A_l} |Q| 2^{dl} = 2^{dl} |X_l|. 
\]
By a property of the Whitney decomposition, there is $C_d>0$ such that
\[
\operatorname{dist}(x, \partial A) \leq C_d 2^{-l}
\]
for all $x\in X_l$. Since $\partial A$ has finite upper Minkowski content, there is $l_1$ only dependent on $A$ such that
\[
|X_l| \leq |\{x: \operatorname{dist}(x, \partial A) \leq C_d 2^{-l}\}| \leq C_A 2^{-l}.
\]
for $l\geq l_1$ which gives us the desired estimate. If $l\leq l_1$ then we simply bound
\[
2^{dl}|X_l| \leq 2^{dl} |A| \leq 2^{(d-1)l} |A| 2^{l_1}. 
\]
\end{proof}
\begin{proof}[Proof of Lemma \ref{shell size boundary}]
For $z\in r \Z^d$ denote by $Q_r(z)$ the cube with side length $r$ centred at $z$. Let $Z_r \subseteq r\Z^d$ be the collection of points $z\in r\Z^d$ such that $Q_r(z) \cap (\partial A + B_r) \neq \varnothing$. If $\omega \in Q_r(z)$ for $ z\in Z_r$, then 
\[
\operatorname{dist}(\omega,\partial A) \leq \inf_{p\in Q_r(z)}(|\omega - p| + \operatorname{dist}(p,\partial A)) \leq (\sqrt{d} +1) r. 
\]
Hence,
\[
\left\lvert  \bigcup_{z\in Z_r} Q_r(z) \right\rvert \leq |\partial A + B_{(\sqrt{d} +1)r}| \leq C r  
\]
for $r$ small enough. Since $Q_r(z)$ for different $z$ are disjoint up to measure $0$, we get $$|Z_r|\le Cr^{1-d}.$$
Finally, we observe that we have the covering $A+B_r \subseteq \cup_{z\in Z_r} Q_r(z)$, so this covering works.

\end{proof}

\begin{remark}
Lemma \ref{shell size} requires only the internal part of the upper Minkowski content of $\partial A$, that is only $|(\partial A+B_r)\cap A|$, but for Lemma \ref{shell size boundary} we need the full upper Minkowski content.
\end{remark}

\begin{remark}
Note that Lemma \ref{shell size boundary} in particular implies that if a set has finite upper Minkowski content then it has finite $d-1$-dimensional Hausdorff measure.
\end{remark}

We will pick a small number $\delta > 0$ and consider the threshold $D = \left[\log_2\left(\frac{c}{\delta \log(\frac{1}{\eps})}\right)\right]$. The cubes with side lengths larger than $2^{-D}$ we will leave as is, and note that their union covers all of $A$ except possibly for the points in $C_d2^{-D}$-neighborhood of $\partial A$. Let $r = C_d2^{-D}$ and consider the cubes $N_r$ from Lemma \ref{shell size boundary}. We get
$$A\subseteq \left(\cup_{l < D}\cup_{Q\in A_l} Q\right) \cup \left(\cup_{Q\in N_r} Q\right).$$
First part of this union is disjoint. We make the second part disjoint by setting, for $X \in N_r,$
\[
X' = X \cap A \setminus (\cup_{l < D}\cup_{Q\in A_l} Q).
\]
Hence, $$A = \left(\cup_{l < D}\cup_{Q\in A_l} Q\right) \cup \left(\cup_{Q\in N_r} Q'\right).$$
From this we can write
\begin{equation}\label{operator decomposition sum}
P_{cA^c}Q_{[0,1]^d}P_{cA} = \sum_{l < D}\sum_{Q\in A_l} P_{cA^c}Q_{[0,1]^d}P_{cQ} + \sum_{Q\in N_r} P_{cA^c}Q_{[0,1]^d}P_{cQ'}.
\end{equation}
Next, we want to enlarge these sets to make the computations easier. To see how the singular values change under enlargements we will use the following lemma.

\begin{lemma}\label{enlargement}
Consider $X, Y, Z \subseteq \R^d$ and assume that $P_X Q_Y P_Z$ is compact. If $X\subset X'$ and $P_{X'}Q_Y P_Z$ is compact, then
\[
\sigma_k(P_X Q_Y P_Z) \leq \sigma_k(P_{X'} Q_Y P_Z).
\]
Similarly, if $Z\subseteq Z'$ and $P_X Q_Y P_{Z'}$ is compact, then 
\[
\sigma_k(P_X Q_Y P_Z) \leq \sigma_k(P_{X} Q_Y P_{Z'}).
\]
\end{lemma}
\begin{proof}
By the max-min theorem, if $T$ is a compact operator on $H$, then
\[
\sigma_k(T) = \max_{\substack{V\subseteq H\\ \dim{H} = k}} \min_{\psi \in V\setminus \{0\}} \frac{\norm{T\psi}}{\norm{\psi}}.
\]
For the first inequality we simply have $\|P_{X'}Q_YP_Z\psi\| \ge \|P_XQ_YP_Z\psi\|$ for all $\psi \in L^2(\R^d)$. The second inequality follows from the first by taking the adjoint. 
\end{proof}

\begin{corollary}\label{Schatten monotonicty}
Let $X,Y,Z$ be measurable sets. If $X \subseteq X'$ and $Z\subseteq Z'$ then $\norm{P_X Q_Y P_Z}_p^p \leq \norm{P_{X'} Q_Y P_{Z'}}_p^p$.
\end{corollary}

We will apply \eqref{quasi norm subadditive} and Corollary \ref{Schatten monotonicty} to \eqref{operator decomposition sum}. For $P_{cA^c}Q_{[0,1]^d}P_{cQ}$, $Q\in A_l$, we enlarge $cA^c$ to $c\tilde{Q}^c$, where $\tilde{Q}$ is the cube with the same centre as $Q$ and two times larger side length. Note that by the property of the Whitney decomposition \ref{our Whitney} $\tilde{Q}$ is still contained in $A$. For $P_{cA^c}Q_{[0,1]^d}P_{cQ'}$, $Q\in N_r$, we simply enlarge $cA^c$ to all of $\R^d$ and $cQ'$ to $cQ$. We get
\begin{equation}\label{triangle for box}
    \|P_{cA^c}Q_{[0,1]^d}P_{cA}\|_p^p \le \sum_{l< D} \sum_{Q\in A_l} \| P_{c \tilde{Q}^c} Q_{[0,1]^d} P_{cQ}\|^p_p +\sum_{Q\in N_r} \| Q_{[0,1]^d} P_{cQ}\|^p_p.
\end{equation}
After shifting with the help of Lemma \ref{change of variables} we can see that the operators in the second sum are simply tensor powers of one-dimensional operators $I_{cr}$ from Lemma \ref{One dimensional operators}. We will use the fact that the Schatten norms are multiplicative under taking the tensor products.
\begin{lemma}\label{tensor product Schatten}
Consider Hilbert spaces $H_1,\dots, H_m$ and compact operators $A_j \in H_j$, $j=1,\dots, m$. If $T$ is the operator $T = A_1 \otimes \dots \otimes A_m$ on $H = H_1\otimes \dots \otimes H_m$, then for any $0<p< \infty$ we have
\[
\| T\|_p = \|A_1\|_p \dots \| A_m\|_p.
\]
\end{lemma}
\begin{proof}
As multisets, we have
\[
\{\sigma_n(T) \mid n\in \N\} = \{ \sigma_{n_1}(A_1) \dots \sigma_{n_m}(A_m) \mid n_1,\dots, n_m\in \N\}.
\]
Summing over $p$'th powers of the elements in the left- and right-hand sides, the claim follows. 
\end{proof}

By Lemma \ref{One dimensional operators} we have
$$\|I_{cr}\|_p^p  \le \sum_{n=1}^{[10cr]}1^p + \sum_{n=[10cr]}^\infty C^p e^{-\tau p n} \le 10 cr + C^p \frac{1}{1-e^{-\tau p}} \le C_1\left(cr + \frac{1}{p}\right)$$
for some absolute constant $C_1$. By our choice of $r$ and $p$ we have $C_2cr \ge \frac{1}{p}$, thus $\|I_{cr}\|_p^p \le C_3 cr$ and therefore by \ref{tensor product Schatten}
\[
\| Q_{[0,1]^d} P_{cQ} \|^p_p = \| I_{cr} \|^{pd}_p \leq C_3^d c^d r^d.
\]

We turn to the operator $P_{c\tilde{Q}} Q_{[0,1]^d} P_{cQ}$ for $Q\in A_l$, $l<D$. To exactly connect it to the one-dimensional operators that we stated in Lemma \ref{One dimensional operators}, we are going to do further simplifications. Firstly, we can, after a translation by means of an argument similar to the proof of Lemma \ref{change of variables}, assume that
\[
Q = I^d , \quad \tilde{Q} = \tilde{I}^d,
\]
where $I = [-2^{-l-1},2^{-l-1}]$ and $\tilde{I} = [-2^{-l},  2^{-l}]$. Then, with the unions being disjoint, 
\[
\tilde{Q}^c = \big(\tilde{I}^{c} \times \R^{d-1}\big) \cup (\tilde{I} \times \tilde{I}^c \times \R^{d-2} \big) \cup \dots \cup \big(\tilde{I}^{d-1} \times \tilde{I}^c\big).
\]
We write $P_{c\tilde{Q}^c}Q_{[0,1]^d}P_Q$ as a sum with respect to this decomposition of $\tilde{Q}^c$ and enlarge $\tilde{I}$ to all of $\R$ in every single instance. It follows by Lemma \ref{enlargement}, Lemma \ref{tensor product Schatten}, and \eqref{quasi norm subadditive} that
\begin{align*}
\|P_{c\tilde{Q}^c} Q_{[0,1]^d} P_{cQ}\|^p_p &\leq d \|\left(Q_{[0,1]} P_{cI}\right)^{\otimes (d-1)} \otimes P_{c\tilde{I}^c} Q_{[0,1]} P_{cI}\|^p_p\\
& = d \| Q_{[0,1]} P_{cI}\|_p^{p(d-1)} \|P_{c\tilde{I}^c} Q_{[0,1]} P_{cI}\|^p_p.
\end{align*}
For $Q_{[0,1]}P_{cI} = I_{c 2^{-l}}$ as above we have
$$\| Q_{[0,1]} P_{cI}\|_p^{p}\le C_1\left(c2^{-l}+\frac{1}{p}\right).$$
Again, by our choice of $D$ and $p$ we always have $C_2 c2^{-l} \ge \frac{1}{p}$, hence $$\| Q_{[0,1]} P_{cI}\|_p^{p}\le C_3 c2^{-l}.$$

For $P_{c\tilde{I}^c} Q_{[0,1]} P_{cI} = J_{c2^{-l-1}}$ it follows from Lemma \ref{One dimensional operators} that
$$\|P_{c\tilde{I}^c} Q_{[0,1]} P_{cI}\|_p^p \le \sum_{n=1}^\infty C^p e^{-\tau p} \le \frac{C_1}{p}.$$
Hence, in total
$$\|P_{c\tilde{Q}^c} Q_{[0,1]^d} P_{cQ}\|^p_p \le \frac{C}{p}c^{d-1}2^{-(d-1)l}.$$

Returning to \eqref{triangle for box}, we see that 
$$\|P_{cA^c}Q_{[0,1]^d}P_{cA}\|^p_p\le C\sum_{l < D} |A_l| c^{d-1}2^{-(d-1)l} + C|N_r|c^d r^d.$$
It follows from Lemma \ref{shell size} that $|A_l| \leq C 2^{(d-1)l}$ and from Lemma \ref{shell size boundary} that $|N_r| \leq C r^{-(d-1)}$. Also, $|A_l| = 0$ for $l < l_0$. Hence
$$\|P_{cA^c}Q_{[0,1]^d}P_{cA}\|^p_p \le \frac{C}{p}\sum_{l_0\le l < D}c^{d-1} + Cc^d r = C\left((D+l_0)c^{d-1}\frac{1}{p} + c^dr\right).$$
By our choice of $D, p, r$ for $\frac{1}{2} > \eps > \alpha^{-c}$ if $\delta$ is small enough we get
\begin{equation}\label{Shatten norm final bound}
\|P_{cA^c}Q_{[0,1]^d}P_{cA}\|^p_p\lesssim c^{d-1}\log\left(\frac{1}{\eps}\right)\log\left(\frac{\alpha c}{\log\left(\frac{1}{\eps}\right)}\right).
\end{equation}
Plugging this estimate into \eqref{final bound Lambda} and using that $N$ is a fixed constant we finally conclude
$$\Lambda_\eps(cA, B)\lesssim c^{d-1}\log\left(\frac{1}{\eps}\right)\log\left(\frac{\alpha c}{\log\left(\frac{1}{\eps}\right)}\right).$$

Next, we prove Theorem \ref{good and good} where both $A$ and $B$ are bounded sets with boundaries of finite upper Minkowski content in the same regime $\frac{1}{2} > \eps > \alpha^{-c}$. For this we do the $S-S^2$ trick again, and this time apply the Whitney decomposition to the set $B$ with the same threshold $D = \left[\log_2\left(\frac{c}{\delta \log(\frac{1}{\eps})}\right)\right]$. With an obvious adaptation of notation from the set $A$, the set $B$ decomposes as 
$$B = \left(\cup_{l < D}\cup_{Q\in B_l} Q\right) \cup \left(\cup_{Q\in N_r} Q'\right).$$
We have 
$$\Lambda_\eps(cA, B) \le \sqrt{e}\|P_{A^c}Q_BP_A\|_p^p \le \sqrt{e}\sum_{l < D}\sum_{Q\in B_l}\|P_{cA^c}Q_QP_{cA}\|_p^p + \sqrt{e}\sum_{Q\in N_r} \|P_{cA^c}Q_{Q'}P_{cA}\|_p^p.$$
For $Q\in B_l$ it follows from Lemma \ref{change of variables} that
$$\|P_{cA^c}Q_QP_{cA}\|_p^p = \|P_{c2^{-l}A^c}Q_{[0,1]^d}P_{c2^{-l}A}\|_p^p,$$
so we can use \eqref{Shatten norm final bound} and get
$$\|P_{cA^c}Q_QP_{cA}\|_p^p \lesssim c^{d-1}2^{-l(d-1)}\log\left(\frac{1}{\eps}\right)\log\left(\frac{\alpha c2^{-l}}{\log\left(\frac{1}{\eps}\right)}\right).$$
For the second logarithm we notice that it is uniformly bounded by $\log\left(\frac{\alpha c2^{-l_0}}{\log\left(\frac{1}{\eps}\right)}\right)$, so, denoting $\alpha' = \alpha 2^{-l_0}$, we get
\begin{align*}
\sum_{l < D}\sum_{Q\in B_l}\|P_{cA^c}Q_QP_{cA}\|_p^p &\lesssim (D+l_0)c^{d-1}\log\left(\frac{1}{\eps}\right)\log\left(\frac{\alpha' c}{\log\left(\frac{1}{\eps}\right)}\right)\\
&\lesssim c^{d-1}\log\left(\frac{1}{\eps}\right)\log^2\left(\frac{\alpha' c}{\log\left(\frac{1}{\eps}\right)}\right),
\end{align*}
which is exactly the required bound.

For the boundary layer $P_{cA^c}Q_{Q'}P_{cA}$, $Q\in N_r$, we first enlarge $cA^c$ to the whole $\R^d$ and enlarge $cA$ to $cR$ where $R$ is a fixed cube containing $A$ to reduce the analysis to the operator $Q_{Q'}P_{cR}$ using Corollary \ref{Schatten monotonicty}. Then, by the equality $\sigma_n(Q_{Q'} P_{cR}) = \lambda_n(cR,Q')^{\frac{1}{2}}$ and the symmetry $\lambda_n(U,V) = \lambda_n(V,U)$ we get
$$\|P_{cA^c}Q_{Q'}P_{cA} \|^p_p \leq \| Q_{Q'} P_{cR}\|_p^p = \|P_{cR} Q_{Q'}P_{cR}\|_{\frac{p}{2}}^{\frac{p}{2}} = \|P_{Q'}Q_{cR}P_{Q'}\|_{\frac{p}{2}}^{\frac{p}{2}}=\|Q_{cR}P_{Q'}\|^p_p.$$
Now, we are able to enlarge $Q'$ to $Q$ and put all of the scaling onto $Q$ by means of Lemma \ref{change of variables} to conclude
$$\|P_{cA^c}Q_{Q'}P_{cA} \|^p_p \leq \|Q_{cR}P_{Q'}\|^p_p\leq \|P_{[0, cr]^d}Q_{[0,1]^d}\|_p^p.$$
By Lemma \ref{tensor product Schatten} the right-hand side is equal to $\|P_{[0, cr]}Q_{[0,1]}\|_p^{pd}$. Arguing  as in the proof of Theorem \ref{union and good} we see
$$\|P_{[0, cr]^d}Q_{[0,1]^d}\|_p^p \lesssim c^d r^d,$$
and therefore
$$\sum_{Q\in N_r} \| P_{cA^c}Q_{Q'}P_{cA}\|^p_p \lesssim |N_r| c^d r^d \lesssim c^dr \lesssim c^{d-1}\log\left(\frac{1}{\eps}\right),$$
which is even smaller than the required bound. This finishes the proof of Theorem \ref{good and good} in the regime $\frac{1}{2} > \eps > \alpha^{-c}$.

Finally, we will deal with the regime $0 < \eps \le \alpha^{-c}$. The proof will be the same for both Theorem \ref{union and good} and Theorem \ref{good and good}. Let $U_A\subset A \subset V_A$ and $U_B \subset B\subset V_B$ be some fixed boxes. By Lemma \ref{enlargement} and symmetry $\lambda_k(X,Y) = \lambda_k(Y,X)$ we have
\[
\lambda_1(cA, B) \leq \lambda_1(cV_A, B) \leq \lambda_1(c V_A, V_b).
\]
In particular, there are no eigenvalues larger than $1-\alpha^{-c}$ for large enough $\alpha$ by Proposition \ref{two boxes}.

For the count of eigenvalues between $\eps$ and $\frac{1}{2}$ we have
$$\Lambda_\eps^-(cA, B) = N_\eps(cA,B)-N_{1/2}(cA,B).$$
For the second term we have
$$N_{1/2}(cA,B)\le 2c^d |A| |B|$$
because ${\rm{Tr}}\,S_{cA,B} = c^d |A| |B|$. By Lemma \ref{enlargement} and symmetry we have
$$\lambda_k(cU_A, U_B)\le \lambda_k(cA,B)\le \lambda_k(cV_A,V_B),$$
and therefore
$$N_\eps(cU_A, U_B)\le N_\eps(cA,B)\le N _\eps(cV_A,V_B).$$
From \eqref{Lambda -, equivalence, two cubes} one can check that if $0 < \eps < \alpha^{-c}$ and $\alpha$ is large enough then
$$N_\eps(cU_A, U_B) \ge 4c^d |A| |B|.$$

Plugging in the bounds we get
$$\frac{N_\eps(cU_a, U_B)}{2} \le N_\eps(cU_A, U_B)-N_{1/2}(cA,B) \le \Lambda_\eps^-(cA, B)\le N_\eps(cV_A, V_B)$$
and both sides are proportional to the required value by \eqref{Lambda -, equivalence, two cubes} when $0 < \eps \le \alpha^{-c}$ for large enough $\alpha$.
\begin{remark}
Note that the final argument gives us a lower bound on eigenvalues for all sets $A, B$ with non-empty interiors and an upper bound on eigenvalues for all bounded sets $A, B$.
\end{remark}

\section{Singular values of one-dimensional operators}
In this section we will prove Lemma \ref{One dimensional operators}. Since the result for $I_r$ can be deduced from Theorem \ref{Bonami2}, we will start with an estimate for $\sigma_k(J_r)$.
\subsection{Singular values of $J_r$}
We will show that the operator $J_r = P_{[-2r,2r]^c} Q_{[0,1]} P_{[-r,r]}$ for $r>0$ 
decomposes into a sum of finite rank operators. This idea is very similar to the one employed by Karnik, Romberg and Davenport \cite{Karnik_Romberg_Davenport} who showed that a similar operator had low displacement-rank structure, and for such operators the singular values can be effectively estimated via Zolotarev numbers. By \eqref{kernel} the operator $J_r$ has kernel
\[
K(x,y) = 1_{[-2r,2r]^c}(x) \check{1}_{[0,1]}(x-y) 1_{[-r,r]}(y) = \frac{1}{2\pi i} 1_{[-2r,2r]^c}(x) \frac{e^{2\pi i(x-y)}-1}{x-y} 1_{[-r,r]}(y).
\]
The variables $x$ and $y$ separate except for the term $\frac{1}{x-y}$. In \cite{Karnik_Romberg_Davenport} the authors handled this term by instead considering $x J_r - J_r x$, which has rank 2. However, simply writing the term $\frac{1}{x-y}$ as a geometric series shows that
\begin{align*}
K(x,y) &= \frac{1}{2\pi i}\sum_{n=0}^\infty \left(x^{-n-1} 1_{[-2r,2r]^c}(x) e^{2\pi i x}\right) \left( y^{n} 1_{[-r,r]^c}(y) e^{-2\pi i y}\right)\\
&\quad\quad \quad -  \left(x^{-n-1} 1_{[-2r,2r]^c}(x)\right) \left( y^{n} 1_{[-r,r]^c}(y)\right).
\end{align*}
Note that the series converges absolutely since $\frac{|y|}{|x|} \leq \frac{1}{2}$ whenever the indicators are non-zero. This is the key point where we use the separation. Hence, if we set 
\begin{align*}
f_n(x) &=  x^{-n-1} 1_{[-2r,2r]^c}(x), \quad \tilde{f}_n(x) = f_n(x) e^{2\pi i x},\\
g_n(y) &=  y^{n} 1_{[-r,r]}(y), \quad \tilde{g}_n(y) = g_n(y) e^{2\pi i y},
\end{align*}
we see that, for $\phi \in L^2$,
\[
J_r \phi = -\frac{1}{2\pi i}\sum_{n=0}^\infty \left(\langle{\phi},g_n\rangle f_n - \langle \phi, \tilde{g}_n \rangle \tilde{f}_n\right).
\]
Clearly each term in the sum has rank at most 2. By a well-known characterization of the singular values of a compact operator $A$,
\[
\sigma_{k+1}(A) = \inf\{ \norm{A - P} \mid P \text{ is an operator of rank at most } k\},
\]
we immediately find, for all $N\in \N$,
\begin{align*}
\sigma_{2N+3}(J_r) \leq \frac{1}{2\pi} \sum_{n=N+1}^\infty \left(\norm{f_n} \norm{g_n} + \lVert \tilde{f}_n \rVert \norm{\tilde{g}_n}\right) = \frac{1}{\pi} \sum_{n=N+1}^\infty \norm{f_n} \norm{g_n}.
\end{align*}
The norms are easily computed exactly:
\begin{align*}
\norm{f_n}^2 = \frac{2}{2n+1} \left(2r\right)^{-2n-1}, \quad \norm{g_n}^2 = \frac{2}{2n+1} r^{2n+1}. 
\end{align*}
Therefore, we get
\[
\sigma_{2N+3}(J_r) \leq \frac{\sqrt{2}}{\pi} \sum_{n=N+1}^\infty \frac{1}{2n+1} 2^{-n} \leq \frac{\sqrt{2}}{\pi} \frac{1}{2N+3} 2^{-N} \le \frac{\sqrt{2}}{\pi}2^{-N},
\]
which finishes the proof.

\subsection{Singular values of $I_r$}
Since $I_r$ is a product of operators with norm at most $1$, it is obvious that $\sigma_n(I_r)\le 1$ for all $n\in\N$, so we will focus on $n > 10r$. As we already mentioned, since $I_r^* I_r = S_{[0,r],[0,1]}$, we can in principle extract the required bound from Theorem \ref{Bonami2}, as well as from many previous results in the literature (in fact, we can replace $10$ by any constant larger than $1$). To keep our argument self-contained we will instead present a complex-analytic proof following the argument in \cite{kulikov_2026}.

Singular values of $I_r$ are square roots of the eigenvalues of $S_{[0,r], [0,1]}$. By Lemma \ref{enlargement} if we increase $r$ the eigenvalues can only increase. Thus, it is enough to consider the case $n = 10r$. Since the eigenvalues depend only on the products of length of the intervals, for convenience we will instead consider the eigenvalues of $S_{[-\frac{1}{2},\frac{1}{2}],[-\frac{r}{2},\frac{r}{2}]}$.

Consider first $n$ normalized eigenfunctions $f_1, f_2, \ldots , f_n$ of $S_{[-\frac{1}{2},\frac{1}{2}],[-\frac{r}{2},\frac{r}{2}]}$ with eigenvalues $\lambda_1(r)\geq \dots \geq \lambda_n(r)$. If $\lambda_n(r) = 0$ then there is nothing to prove (in fact, it is not hard to see again by the max-min principle that this is never the case), so we assume without loss of generality that $\lambda_n(r)>0$. It follows immediately that $f_1,\dots, f_n$ are supported on $[-\frac{1}{2},\frac{1}{2}]$. In particular, their Fourier transforms
$$g_k(z) = \hat{f}_k(z) =\int_{-\frac{1}{2}}^\frac{1}{2} f_k(t)e^{-2\pi i zt}dt$$
are defined for all $z\in\Cm$ and they are holomorphic functions of $z$. 

Let us fix $n-1$ distinct complex numbers $z_1, z_2, \ldots , z_{n-1}$. By simple linear algebra we can find scalars $a_1,\dots, a_n \in \mathbb{C}$ not all zero such that the linear combination $g(z) = \sum_{k=1}^n a_kg_k(z)$ satisfies $g(z_1) = \dots = g(z_{n-1}) = 0$. By scaling we can assume that $\|g\|_2 = 1$.

First, we have a uniform pointwise bound on $g(x+iy)$:
\begin{equation}\label{g uniform bound}
|g(x+iy)| \le e^{\pi |y|}.
\end{equation}
Indeed, putting $f = \sum a_k f_k$, we have
$$|g(x+iy)| = \left|\int_{-\frac{1}{2}}^\frac{1}{2} f(t)e^{-2\pi i (x+iy)t}dt\right| \le e^{\pi |y|}\int_{-\frac{1}{2}}^\frac{1}{2} |f(t)|dt \le e^{\pi |y|},$$
where in the last step we used the Cauchy--Schwarz inequality. By the max-min characterization of the eigenvalues we also have 
\begin{equation}\label{g lower bound}
\lambda_n(r)^2 \leq \|S_{[-\frac{1}{2},\frac{1}{2}],[-\frac{r}{2},\frac{r}{2}]} f\|^2 \leq \int_{-r/2}^{r/2} \abs{g(x)}^2 \, dx.
\end{equation}
We want to arrive at a contradiction by choosing an appropriate sequence $z_l$.

We will assume for now that $n\ge 1000$. The cases $n < 1000$ we will cover at the end. Our sequence will be an arithmetic progression with $z_1 = -r$, $z_{n-1} = r$. Its step is $s=\frac{2r}{n-1}$. Since $n = 10r$, $n\ge 1000$ we have $s \le \frac{1}{4}$. We want to bound $|g(x_0)|$ for $x_0\in [-\frac{r}{2},\frac{r}{2}]$ assuming that $|g(x+iy)|\le e^{\pi |y|}$ for all $x, y\in \R$ and $g(z_l) = 0$. We will use Jensen's formula for the disk centred at $x_0$ of radius $\frac{r}{2}$:

$$\log |g(x_0)| = \int_0^1 \log \left|g\left(x_0 + \frac{r}{2}e^{2\pi i t}\right)\right|dt + \sum_{|z-x_0| < \frac{r}{2}, g(z)=0} \log \frac{2|z-x_0|}{r}.$$

Since each term in the sum is negative, by leaving in it only $z_l$'s we get an inequality
\begin{equation}\label{Jensen}
\log |g(x_0)| \le \int_0^1 \log \left|g\left(x_0 + \frac{r}{2}e^{2\pi i t}\right)\right|dt + \sum_{|z_l-x_0| < \frac{r}{2}} \log \frac{2|z_l-x_0|}{r}.
\end{equation}
For the integral term we have $\left|g\left(x_0 + \frac{r}{2}e^{2\pi i t}\right)\right| \le e^{\pi \frac{r}{2}|\sin(2\pi t)|}$ by \eqref{g uniform bound}, hence 
$$\int_0^1 \log \left|g\left(x_0 + \frac{r}{2}e^{2\pi i t}\right)\right|dt \le \int_0^1\frac{\pi r}{2} |\sin(2\pi t)|dt = r.$$
For the sum over $z_l$'s to the right of $x_0$ we have
\begin{align*}
\sum_{x_0 < z_l< x_0 + \frac{r}{2}} \log \frac{2|z_l-x_0|}{r} &\leq \frac{1}{s} \sum_{x_0 < z_l< x_0 + \frac{r}{2}} \int_{z_l}^{z_l + s}  \log \frac{2(\omega-x_0)}{r} d\omega \\
&= \frac{1}{s}\int_{\min\limits_{z_l > x_0} z_l}^{\max\limits_{z_l < x_0+\frac{r}{2}} z_l+s} \log \frac{2(\omega-x_0)}{r} d\omega \\
&\le\frac{1}{s}\int_{s}^{\frac{r}{2}+s}  \log \frac{2\omega}{r} d\omega \leq - \frac{1}{4s} r + 1.
\end{align*}
where in the first step we used monotonicity of the function $\log \frac{2t}{r}$ for $t > 0$, in the second step we combined the integrals into one, in the third step we used that the function $\log \frac{2t}{r}$ is negative for $t < \frac{r}{2}$ and positive for $t > \frac{r}{2}$ and that the minimum of $z_l$'s is at most $x_0+s$ and the maximum of $z_l$'s is at least $x_0+\frac{r}{2}-s$ because $x_0 + \frac{r}{2} \le r$, so we can not get past the interval $[-r.r]$, and in  the last step we used inequality $\log u \le u-1$. For the sum over $z_l$'s to the left of $x_0$ we get exactly the same upper bound. 

Plugging all our bounds into \eqref{Jensen} together with $s\le \frac{1}{4}$ we get
$$|g(x_0)|\le e^{2-r}.$$
Therefore, by \eqref{g lower bound}
$$\lambda_n(r)^2\le\int_{-\frac{r}{2}}^\frac{r}{2} |g(x)|^2dx \le re^{4-2r} \le e^{4-r},$$
as required. 

For $n \le 1000$ we can simply estimate $\lambda_n(r) \le 1$, which is consistent with our bound if we increase $C$.

\section{Trace class conditions}
In this section we prove Theorem \ref{trace class sufficient} and Theorem \ref{trace class necessary}. 

\subsection{Sufficient condition}
We begin with Theorem \ref{trace class sufficient}, so consider bounded measurable sets $A,B\subseteq \R^d$ and let $f$ be a trace class admissible function, that is
\begin{equation}\label{f condition again}
\int_0^{\delta} \frac{M_0 f(\eps) \log\left(\frac{1}{\eps}\right)^{d-1}}{\eps\left(\log\log\left(\frac{1}{\eps}\right)\right)^d} \, d\eps < \infty.
\end{equation}

Since $A$ and $B$ are bounded they have finite measure and so $S_{A,B}$ is a compact operator. We need to show that $\sum_{n=1}^\infty |f(\lambda_n(A,B))| < \infty$. Since $\lambda_n(A,B)\to 0$ there are only finitely many $n$ such that $\lambda_n(A,B) > \frac{\delta}{2}$. So, it is enough to show that the tail $\sum_{\lambda_n(A,B)\leq \frac{\delta}{2}}|f(\lambda_n(A,B))|$ is finite. We will also ignore all of the eigenvalues equal to $0$ since $f(0) = 0$ for all trace class admissible for $L^2(\R^d)$ functions $f$.
We decompose the sum dyadically and bound $f$ from above by $M_0f$:
\begin{align*}
\sum_{0<\lambda_n(A,B)\leq \frac{\delta}{2}}|f(\lambda_n(A,B))| &= \sum_{k=0}^\infty \sum_n 1_{\{\delta 2^{-2^{k+1}}<\lambda_n(A,B) \le \delta 2^{-2^k} \}} \, |f(\lambda_n(A,B))|\\
& \leq \sum_{k=0}^\infty M_0 f\left(\delta 2^{-2^k}\right)\left(N_{\delta 2^{-2^{k+1}}}(A,B) - N_{\delta 2^{-2^k}}(A,B)\right)\\
&\le \sum_{k=0}^\infty M_0 f\left(\delta 2^{-2^k}\right)N_{\delta 2^{-2^{k+1}}}(A,B).
\end{align*}
We chose this exact decomposition because on this scale the estimates in \eqref{Lambda -, equivalence, two cubes} roughly double at each step. 
Since $A$ and $B$ are bounded, there exist fixed cubes $V_A, V_B$ such that $A\subseteq V_A$, $B\subseteq V_B$. By Lemma \ref{enlargement} we have $N_\eps(A,B)\le N_\eps(V_A,V_B)$ for all $0 < \eps < 1$. Hence,
\begin{align*}
\sum_{k=0}^\infty M_0 f\left(\delta 2^{-2^k}\right)N_{\delta 2^{-2^{k+1}}}(A,B)&\leq \sum_{k=0}^\infty M_0 f\left(\delta 2^{-2^k}\right)N_{\delta 2^{-2^{k+1}}}(V_A,V_B)\\
&= \sum_{k=0}^\infty M_0 f\left(\delta 2^{-2^k}\right)(N_{1/2}(V_A,V_B)+\Lambda^-_{\delta 2^{-2^{k+1}}}(V_A,V_B)).
\end{align*}
Since the integral \eqref{f condition again} converges, $M_0  f(\eps)$ is finite for all $\eps < \delta$. In particular, it suffices to show that the tail
\[
\sum_{k=k_0}^\infty M_0 f\left(\delta 2^{-2^k}\right)(N_{1/2}(V_A,V_B)+\Lambda^-_{\delta 2^{-2^{k+1}}}(V_A,V_B))
\]
is finite for some $k_0 >0$. Note that $N_{1/2}(V_A,V_B)$ is a fixed number while $\Lambda^-_{\delta 2^{-2^{k+1}}}(V_A,V_B)$ tends to infinity, so for $k_0$ large enough we have $N_{1/2}(V_A,V_B) \leq \Lambda^-_{\delta 2^{-2^{k+1}}}(V_A,V_B)$ for $k\geq k_0$. By possibly choosing $k_0$ larger still we can apply the upper bound \eqref{Lambda -, equivalence, two cubes} (note that $c$, which is the product of side lengths of $V_A$ and $V_B$, is a constant for the present discussion), and find
\begin{align*}
\MoveEqLeft \sum_{k=k_0}^\infty M_0 f\left(\delta 2^{-2^k}\right)(N_{1/2}(V_A,V_B)+\Lambda^-_{\delta 2^{-2^{k+1}}}(V_A,V_B))\\
&\lesssim \sum_{k=k_0}^\infty M_0 f\left(\delta 2^{-2^k}\right) \left( \frac{\log\left(\tfrac{1}{\delta 2^{-2^{k+1}}}\right)}{\log \left(\tfrac{\log\left(\frac{1}{\delta 2^{-2^{k+1}}}\right)}{c}\right)}\right)^d.
\end{align*}
A direct computation shows that the expression in brackets is proportional to $\frac{2^k}{k}$. So, it remains to show that $\sum_{k=k_0}^\infty M_0 f\left(\delta 2^{-2^k}\right) \frac{2^{kd}}{k^d}$ is finite. We have the following chain of inequalities:
\begin{align*}
\sum_{k=k_0}^\infty M_0 f\left(\delta 2^{-2^k}\right) \frac{2^{kd}}{k^d} &= \sum_{k=k_0}^\infty M_0 f\left(\delta 2^{-2^k}\right) \frac{2^{kd}}{k^d}  \frac{1}{\log(2) 2^{k-1}} \int_{\delta 2^{-2^k}}^{\delta 2^{-2^{k-1}}} \frac{1}{\eps} \, d\eps\\
&\lesssim \sum_{k=k_0}^\infty  M_0 f\left(\delta 2^{-2^k}\right) \int_{\delta 2^{-2^k}}^{\delta 2^{-2^{k-1}}} \frac{\log\left(\frac{1}{\eps}\right)^{d-1}}{\eps\left(\log\log\left(\frac{1}{\eps}\right)\right)^d} \, d\eps\\
&\leq \sum_{k=k_0}^\infty \int_{\delta 2^{-2^k}}^{\delta 2^{-2^{k-1}}} \frac{M_0 f(\eps) \log\left(\frac{1}{\eps}\right)^{d-1}}{\eps\left(\log\log\left(\frac{1}{\eps}\right)\right)^d} \, d\eps\\
& \leq \int_0^{\delta} \frac{M_0 f(\eps) \log\left(\frac{1}{\eps}\right)^{d-1}}{\eps\left(\log\log\left(\frac{1}{\eps}\right)\right)^d} \, d\eps < \infty.
\end{align*}
We used that for $\delta 2^{-2^k} < \eps < \delta 2^{-2^{k-1}}$ the term $\frac{\log\left(\frac{1}{\eps}\right)^{d-1}}{\left(\log\log\left(\frac{1}{\eps}\right)\right)^d}$ is proportional to $\frac{2^{k(d-1)}}{k^d}$ and that $M_0 f$ is non-decreasing. This finishes the proof of Theorem \ref{trace class sufficient}.

\subsection{Necessary condition}
We turn to the proof of Theorem \ref{trace class necessary}, so assume that $A, B \subseteq \R^d$ are sets with non-empty interiors such that $S_{A,B}$ is compact, and consider a function $f$ such that $\abs{f(x)}$ is non-decreasing close to $0$, say for $0\le x\le\beta$. Assume that $f(S_{A,B})$ is trace class. We need to show that 
\[
\int_0^{\delta} \frac{M_0 f(\eps) \log\left(\frac{1}{\eps}\right)^{d-1}}{\eps\left(\log\log\left(\frac{1}{\eps}\right)\right)^d} \, d\eps < \infty
\]
for some $\delta > 0$. Note that for $0 \le \eps \le \beta$ we have $M_0f(\eps) = |f(\eps)|$, so if $\delta \le \beta$ then 
\[
\int_0^{\delta} \frac{M_0 f(\eps) \log\left(\frac{1}{\eps}\right)^{d-1}}{\eps\left(\log\log\left(\frac{1}{\eps}\right)\right)^d} \, d\eps = \int_0^{\delta} \frac{|f(\eps)| \log\left(\frac{1}{\eps}\right)^{d-1}}{\eps\left(\log\log\left(\frac{1}{\eps}\right)\right)^d} \, d\eps.
\]
As in the proof of Theorem \ref{trace class sufficient}, we split $\sum |f(\lambda_n(A,B))|$ into the sum over eigenvalues larger than $\beta$ and at most $\beta$ and throw away the first one as it is finite. Thus, our basic assumption is that
$$\sum_{\lambda_n(A,B)\leq \beta}|f(\lambda_n(A,B))| < \infty.$$
By definition $\lambda_n(A,B)$ is non-increasing in $n$ so this is the sum from some $n_0$ to infinity. Since $A$ and $B$ have non-empty interiors, we can find cubes $U_A$ and $U_B$ such that $U_A\subset A$, $U_B\subset B$. By Lemma \ref{enlargement} we have $\lambda_n(A,B)\ge \lambda_n(U_A, U_B)$. Using monotonicity of $|f|$ we get
$$\sum_{\lambda_n(A,B)\leq \beta}|f(\lambda_n(A,B))|  \ge \sum_{n=n_0}^\infty |f(\lambda_n(U_A, U_B))|.$$

Next, we do the dyadic decomposition but with a large constant $K\geq 2$ to be fixed later. Pick $k_0$ so that $2^{-K^{k_0}} \leq \lambda_{n_0}(U_A,U_B)$ and write, using 
monotonicity of $|f|$,
\begin{align*}
\sum_{n=n_0}^\infty |f(\lambda_n(U_A, U_B))| &\ge \sum_{k=k_0}^\infty \sum_n 1_{\{ 2^{-K^{k+1}}<\lambda_n(U_A,U_B) \le 2^{-K^k} \}} \, |f(\lambda_n(U_A,U_B))|\\
&\geq \sum_{k=k_0}^\infty \left|f\left(2^{-K^{k+1}}\right)\right| \left( N_{2^{-K^{k+1}}}(U_A, U_B) - N_{2^{-K^{k}}}(U_A, U_B)\right)\\
& = \sum_{k=k_0}^\infty \left|f\left(2^{-K^{k+1}}\right)\right| \left( \Lambda^-_{2^{-K^{k+1}}}(U_A, U_B) - \Lambda^-_{2^{-K^{k}}}(U_A, U_B)\right)
\end{align*}

We intend to apply \eqref{Lambda -, equivalence, two cubes}. Note here that there are $c_1, c_2>0$ independent of $K$ such that for all $k \geq k_0(K)$
$$c_1\frac{K^{kd}}{(\log(K)k)^d}\le \Lambda^-_{2^{-K^{k}}}(U_A, U_B) \le c_2 \frac{K^{kd}}{(\log(K)k)^d}.$$
We choose $K$ so that $K^d \ge \frac{2^{d+1}c_2}{c_1}$ which fixes $k_0$ and ensures that 
\[
\Lambda^-_{2^{-K^{k+1}}}(U_A, U_B) \ge 2\Lambda^-_{2^{-K^{k}}}(U_A, U_B)
\]
whenever $k\geq k_0$. Hence, 
\begin{align*}
\MoveEqLeft \sum_{k=k_0}^\infty \left|f\left(2^{-K^{k+1}}\right)\right| \left( \Lambda^-_{2^{-K^{k+1}}}(U_A, U_B) - \Lambda^-_{2^{-K^{k}}}(U_A, U_B)\right) \\
&\geq  \sum_{k=k_0}^\infty \left|f\left(2^{-K^{k+1}}\right)\right| \Lambda^-_{2^{-K^{k}}}(U_A, U_B)\\
&\gtrsim \sum_{k=k_0}^\infty \left|f\left(2^{-K^{k+1}}\right)\right|   \frac{K^{kd}}{k^d}.
\end{align*}
Arguing as in the proof of Theorem \ref{trace class sufficient} and using monotonicity of $|f|$ we obtain
$$\sum_{k=k_0}^\infty \left|f\left(2^{-K^{k+1}}\right)\right|   \frac{K^{kd}}{k^d} \gtrsim\int_0^{2^{-K^{k+1}}}\frac{|f(\eps)| \log\left(\frac{1}{\eps}\right)^{d-1}}{\eps\left(\log\log\left(\frac{1}{\eps}\right)\right)^d} \, d\eps.$$
Collecting everything, we conclude
\[
\infty > \sum_{\lambda_n(A,B)\leq \beta}|f(\lambda_n(A,B))| \gtrsim \sum_{k=k_0}^\infty \left|f\left(2^{-K^{k+1}}\right)\right|   \frac{K^{kd}}{k^d} \gtrsim\int_0^{2^{-K^{k+1}}}\frac{|f(\eps)| \log\left(\frac{1}{\eps}\right)^{d-1}}{\eps\left(\log\log\left(\frac{1}{\eps}\right)\right)^d} \, d\eps,
\]
finishing the proof.

\section{Area laws}
In this section we prove Theorem \ref{good area law} and Theorem \ref{bad area law}.

\subsection{Uniform bounds}
We now prove \eqref{area law uniform bound} and Theorem \ref{bad area law}. For both of them we can without loss of generality assume that $f(1) = 0$. Indeed, denoting $\tilde{f}(\theta) = f(\theta) - \theta f(1)$ we have $\tilde{f}(1) = 0$, $${\rm Tr}f(S_{cA,B}) = {\rm Tr}\tilde{f}(S_{cA,B}) + f(1){\rm Tr}(S_{cA,B}) ={\rm Tr}\tilde{f}(S_{cA,B}) +c^d|A||B|f(1)$$ 
and 
\[
\int_0^1 \frac{\tilde{f}(\theta) - \theta\tilde{f}(1)}{\theta(1-\theta)} \, d\theta = \int_0^1 \frac{f(\theta) - f(1)\theta}{\theta(1-\theta)} \, d\theta.
\]
Note that $\tilde{f}$ is also area law admissible and trace class admissible for $L^2(\R^d)$ if $f$ is. Thus, it suffices to establish the behaviour of ${\rm Tr}\tilde{f}(S_{cA,B})$ for \eqref{area law two term}. To get the claimed uniform bound \eqref{area law uniform bound} we also need to show that $$ \int_0^1 \frac{M_0\tilde{f}(\eps) + M_1\tilde{f}(\eps)}{\eps}d\eps \leq C\int_0^1 \frac{M_0f(\eps) + M_1f(\eps)}{\eps}d\eps$$ for some absolute constant $C > 0$. We have $M_0 \tilde{f}(\eps) \le M_0 f(\eps) + \eps|f(1)|$ and $M_1\tilde{f}(\eps) \le M_1 f(\eps) + \eps |f(1)|$, so we just have to show that $$\lambda\int_0^1 \frac{M_0f(\eps) + M_1f(\eps)}{\eps}d\eps \ge |f(1)|$$ for some absolute constant $\lambda > 0$. For $\eps > \frac{1}{2}$ we have $M_0f(\eps) \ge |f(\frac{1}{2})|$ and $M_1f(\eps) \ge |f(1) - f(\frac{1}{2})|$, thus $M_0f(\eps) +M_1f(\eps)\ge |f(1)|$. Integrating this we get
$$\int_0^1\frac{M_0f(\eps) + M_1f(\eps)}{\eps}d\eps \ge \log(2)|f(1)|,$$
so $\lambda = \frac{1}{\log(2)}$ works.

We start with proving the uniform bound \eqref{area law uniform bound}, so assume that $A\subseteq \R^d$ is a set whose boundary $\partial A$ has finite upper Minkowski content and $B\subseteq\R^d$ is a finite union of parallelepipeds with disjoint interiors, and let $f$ be both area law admissible and trace class admissible for $L^2(\R^d)$ with $f(1) = 0$. We need to show that, for $c$ large enough dependent on $A$, $B$, and $f$,
\[
|\operatorname{Tr} f(S_{cA, B})| \leq  C(A,B)c^{d-1}\log(c) \int_{0}^1 \frac{M_0 f(\eps) + M_1 f(\eps)}{\eps} \, d\eps
\]
The proof of Theorem \ref{bad area law} will be almost identical. Take $\alpha$ from the statement of Theorem \ref{union and good}. We split the eigenvalues into the relevant regimes:
\begin{equation}\label{three sums}
\begin{aligned}
|\operatorname{Tr} f(S_{cA, B})| &= \left|\sum_{n=1}^\infty f(\lambda_n(cA,B))\right| \le \sum_{n = 1}^\infty |f(\lambda_n(cA,B))|\\
&=\sum_{\lambda_n(cA,B) >1-\alpha^{-c}}|f(\lambda_n(cA,B))|+\sum_{1-\alpha^{-c}\ge\lambda_n(cA,B) \ge\alpha^{-c}}|f(\lambda_n(cA,B))|\\
&\quad +\sum_{\alpha^{-c}>\lambda_n(cA,B) >0}|f(\lambda_n(cA,B))|.
\end{aligned}
\end{equation}
We will bound each of these three sums separately. The first sum is empty by Theorem \ref{union and good}, since there are no eigenvalues larger than $1-\alpha^{-c}$. We proceed with the bound for the second sum. We split the sum into eigenvalues close to 1 and close to 0 and decompose dyadically like in the proof of Theorem \ref{trace class sufficient}. Let $k_0(c)$ be such that $2^{-2^{k_0(c)+1}} < \alpha^{-c} \le 2^{-2^{k_0(c)}}$. We write
\begin{align*}
 \sum_{1-\alpha^{-c}\ge\lambda_n(cA,B)\ge \alpha^{-c}}|f(\lambda_n(cA,B))| &\leq \sum_{k=0}^{k_0(c)} \sum_{n}  1_{\{ 2^{-2^{k+1}}<\lambda_n(cA,B) \le 2^{-2^k} \}} \,  |f(\lambda_n(cA,B))|\\
 &\quad+ \sum_{k=0}^{k_0(c)} \sum_{n}  1_{\{ 2^{-2^{k+1}}<1-\lambda_n(cA,B) \le 2^{-2^k} \}} \,  |f(\lambda_n(cA,B))|\\
 &\leq   \sum_{k=0}^{k_0(c)} \sum_{n}  1_{\{ 2^{-2^{k+1}}<\lambda_n(cA,B) \le 2^{-2^k} \}} \,  M_0 f\left(2^{-2^k}\right)\\
 &\quad+ \sum_{k=0}^{k_0(c)} \sum_{n}  1_{\{ 2^{-2^{k+1}}<1-\lambda_n(cA,B) \le 2^{-2^k} \}} \,  M_1 f\left(2^{-2^k}\right)\\
 &\le \sum_{k=0}^{k_0(c)} \left(M_0 f\left(2^{-2^k}\right) + M_1 f\left(2^{-2^k}\right)\right)\Lambda_{2^{-2^{k+1}}}(cA,B).
\end{align*}
In the first step we did the dyadic splitting (with possibly overcounting beyond $\alpha^{-c}$), in the second step we bounded $f$ by $M_0 f$ and $M_1 f$, respectively, and in the third step we crudely bounded the number of eigenvalues in the corresponding intervals by $\Lambda_{\eps}(cA,B)$ for a suitable $\eps$.

For $k < k_0(c)$ we will use \eqref{eps not tiny}. For $k = k_0(c)$ we have to use \eqref{eps tiny} but one can check that in this regime the estimate is proportional to the one in \eqref{eps not tiny} so we will use \eqref{eps not tiny} here as well. We get
\begin{equation}\label{what changes}
\Lambda_{2^{-2^{k+1}}}(cA,B) \lesssim c^{d-1} 2^{k} \log\left( \alpha c2^{-k} \right) \lesssim c^{d-1}2^k \log(c). 
\end{equation}
Applying this we find 
\begin{align*}
\MoveEqLeft \sum_{1-\alpha^{-c}\ge\lambda_n(cA,B)\ge \alpha^{-c}}|f(\lambda_n(cA,B))| \leq \sum_{k=0}^{k_0(c)} \left(M_0 f\left(2^{-2^k}\right) + M_1 f\left(2^{-2^k}\right)\right)\Lambda_{2^{-2^{k+1}}}(cA,B)\\
&\lesssim \sum_{k=0}^{k_0(c)} \left(M_0 f\left(2^{-2^k}\right) + M_1 f\left(2^{-2^k}\right)\right) c^{d-1} 2^k \log(c).
\end{align*}
By the simple identity 
\[
2^{k} = \frac{2}{\log(2)}\int_{2^{-2^{k}}}^{2^{-2^{k-1}}} \frac{1}{\eps} \, d\eps,
\]
we conclude 
\begin{align*}
\MoveEqLeft \sum_{1-\alpha^{-c}\ge\lambda_n(cA,B)\ge \alpha^{-c}}|f(\lambda_n(cA,B))|\\
&\lesssim \sum_{k=0}^{k_0(c)} \left(M_0 f\left(2^{-2^k}\right) + M_1 f\left(2^{-2^k}\right)\right) c^{d-1} \log(c) \frac{2}{\log(2)}\int_{2^{-2^{k}}}^{2^{-2^{k-1}}} \frac{1}{\eps} \, d\eps\\
& \leq c^{d-1}\log(c) \frac{2}{\log(2)}  \sum_{k=0}^{k_0(c)} \int_{2^{-2^{k}}}^{2^{-2^{k-1}}} \frac{M_0 f(\eps) + M_1 f( \eps)}{\eps} \, d\eps\\
& \leq c^{d-1}\log(c) \frac{2}{\log(2)}\int_0^{2^{-1/2}} \frac{M_0 f(\eps)+M_1 f(\eps)}{\eps}d\eps,
\end{align*}
which is of the required form for \eqref{area law uniform bound}.

We finally turn to the third sum in \eqref{three sums}. Using the exact same dyadic decomposition we will ultimately get
\begin{equation}\label{log log or log 2 or loglog2}
\sum_{\alpha^{-c}>\lambda_n(cA,B) >0}|f(\lambda_n(cA,B))| \le \sum_{k=k_0(c)}^\infty M_0 f\left(2^{-2^k}\right) \Lambda_{2^{-2^{k+1}}}^-(cA,B).
\end{equation}
Since we are now in the regime $\eps < \alpha^{-c}$, it follows from \eqref{eps tiny} that
$$\Lambda_{2^{-2^{k+1}}}^-(cA,B) \lesssim  \left(\frac{2^{k+1}}{\log(\frac{2^{k+1}}{c})}\right)^d.$$
 We claim that $\left(\frac{2^{k+1}}{\log\left(\frac{2^{k+1}}{c}\right)}\right)^d\le \frac{2^{(k+1)d}}{(k+1)^d}C(d,\alpha)\log(c)^d$ for $k \ge k_0(c)$, where $C(d,\alpha)$ is some constant depending only on $d$ and $\alpha$. Indeed, this is equivalent to $\frac{k+1}{\log\left(\frac{2^{k+1}}{c}\right)} \le C(d,\alpha)^{1/d}\log(c)$. The left-hand side is equal to $\frac{1}{\log(2)-\frac{\log(c)}{k+1}}$ which is clearly a decreasing function of $k$, so it is enough to verify the inequality for $k = k_0(c)$. Using $2^{-2^{k_0(c)+1}}\le \alpha^{-c}\le 2^{-2^{k_0(c)}}$ and $\alpha \ge 4$ we get 
$$\frac{k_0(c)+1}{\log\left(\frac{2^{k_0(c)+1}}{c}\right)}  \le \frac{\log_2(c)+\log_2 (\log_2(\alpha))+1}{\log(\log_2(\alpha))}\le \log(c) \frac{\frac{1}{\log(2)}+\log_2(\log_2(\alpha)) + 1}{\log(2)}.$$
which gives the claim with $C(d,\alpha)^{1/d}=\frac{\frac{1}{\log(2)}+\log_2(\log_2(\alpha)) + 1}{\log(2)}$.

Plugging this into \eqref{log log or log 2 or loglog2} and arguing as in the proof of Theorem \ref{trace class sufficient} we get
\begin{equation}\label{trace class condition bound}
\begin{aligned}\sum_{\alpha^{-c}>\lambda_n(cA,B) >0}|f(\lambda_n(cA,B))| &\lesssim \log(c)^d\sum_{k=k_0(c)}^\infty  M_0 f\left(2^{-2^k}\right) \frac{2^{(k+1)d}}{(k+1)^d}\\
&\lesssim \log(c)^d\int_0^{2^{-2^{k_0(c)-1}}}\frac{M_0 f(\eps)\log\left(\frac{1}{\eps}\right)^{d-1}}{\eps \left(\log\log\left(\frac{1}{\eps}\right)\right)^d} \, d\eps.
\end{aligned}
\end{equation}
The final integral is finite for $c$ large enough since $f$ is trace class admissible for $L^2(\R^d)$. Moreover, as $c\to \infty$ we have $k_0(c)\to \infty$, thus the integral can be as small as we like. In particular, for $c > c_0(f, A, B)$ we can assume that 
$$ \int_0^{2^{-2^{k_0(c)-1}}}\frac{M_0 f(\eps)\log\left(\frac{1}{\eps}\right)^{d-1}}{\eps \left(\log\log\left(\frac{1}{\eps}\right)\right)^d} \, d\eps\leq \int_0^1 \frac{M_0 f(\eps) + M_1 f(\eps)}{\eps}d\eps$$
assuming that $\int_0^1 \frac{M_0 f(\eps) + M_1 f(\eps)}{\eps}d\eps$ is non-zero. But it can be zero only if $f\equiv 0$ in which case the theorem is trivial. Thus, for $c > c_0(f, A, B)$ we have
$$|{\rm Tr} f(S_{cA,B})|\le C(A,B)(c^{d-1}\log(c)+\log(c)^d)\int_0^1\frac{M_0f(\eps)+M_1f(\eps)}{\eps}d\eps.$$
For $c > 1$ we have $\log(c) \le c$, which gives the desired estimate with at most doubling the constant $C(A,B)$.

To prove Theorem \ref{bad area law} the only thing that we have to change is that in \eqref{what changes} we will have $\log^2(c)$ instead of $\log(c)$, which leads to the final error bound $O(c^{d-1}\log^2(c))$.
\begin{remark}
 It was absolutely crucial for our argument that the third sum in \eqref{three sums} turned out to be $o(c^{d-1}\log(c))$, otherwise it might happen that the bound does not hold. For $d\ge 2$ the third sum is $O(\log(c)^d) = o(c^{d-1}\log(c))$, so in this case it is enough for us to only know the value of $\delta$ and the integral in the definition of trace class admissibility for $L^2(\R^d)$. 
 
 For $d = 1$ our proof as written requires us to also know how fast does the integral $\int_0^t \frac{M_0f(\eps)}{\eps \log\log(\frac{1}{\eps})}d\eps$ converges to $0$ as $t\to 0$ to get $O(\log(c))$ with as small of a constant as we like. However, coincidentally $d = 1$ is also the only case where trace class admissibility for $L^2(\R^d)$ is weaker than area law admissibility. In particular, 
 $$\int_0^t \frac{M_0f(\eps)}{\eps \log\log(\frac{1}{\eps})}d\eps \le \frac{1}{\log\log(\frac{1}{t})}\int_0^t \frac{M_0f(\eps)}{\eps}d\eps.$$
 In this way we would get that the third sum is $O(1)$ (and even $o(1)$) for $d = 1$, and the value $c_0(f, A, B)$ would even be independent of $f$.
\end{remark}
\subsection{Two-term asymptotics}
It remains to establish the two-term asymptotic expansion \eqref{area law two term} under the same assumptions on $A$ and $B$ as in the previous subsection, but with the additional assumption that $f$ is Riemann integrable on $[\eps, 1-\eps]$ for all $0<\eps<1/2$. We will also assume without loss of generality that $f$ is real-valued, as we can first prove the result for ${\rm Re}f$ and ${\rm Im}f$ separately, which satisfy all of our assumptions if $f$ satisfies them, and use that both sides of \eqref{area law two term} are linear in $f$. Lastly, as before, we will also assume that $f(1) = 0$ by subtracting $\theta f(1)$ from $f$.

Recall that \eqref{area law two term} holds for polynomials \cite{EE_paper}. We will first extend \eqref{area law two term} to all continuous functions supported on $[\eps, 1-\eps]$ for some $\eps > 0$ and then deduce from this that \eqref{area law two term} holds for $f = 1_{[a,b]}$ for all $0 < a < b < 1$. This is the same argument that was used in \cite{widom_landau} to prove Theorem \ref{Slepian}.

Let $f:[0,1]\to \R$ be a continuous function such that $f(\theta) = 0$ if $0 < \theta < \eps$ or $1-\eps < \theta < 1$. The function $g(\theta) = \frac{f(\theta)}{\theta(1-\theta)}$ is clearly continuous on $[0, 1]$. Given $\delta > 0$, by the Stone--Weierstrass theorem, we can find a polynomial $P$ such that $|P(\theta)-g(\theta)| \le \delta$ for all $\theta\in [0,1]$. We clearly have $$(P(\theta)+\delta)\theta(1-\theta) \ge f(\theta)\ge(P(\theta)-\delta)\theta(1-\theta).$$
Denoting $Q(\theta)=(P(\theta)+\delta)\theta(1-\theta)$, $R(\theta)=(P(\theta)-\delta)\theta(1-\theta)$ we have
\[
\operatorname{Tr} Q(S_{cA,B}) \geq \operatorname{Tr} f(S_{cA,B}) \geq \operatorname{Tr} R(S_{cA,B})
\]
Dividing by $c^{d-1}\log(c)$, applying \eqref{area law two term} to $Q$ and $R$ and taking the limit $c\to \infty$ we get
\begin{align*}
I(A,B)\int_0^1 \frac{Q(\theta)}{\theta(1-\theta)}d\theta &\ge \limsup_{c\to \infty}\frac{\operatorname{Tr} f(S_{cA,B})}{c^{d-1}\log(c)}\\
&\ge \liminf_{c\to\infty} \frac{\operatorname{Tr} f(S_{cA,B})}{c^{d-1}\log(c)}\ge I(A,B)\int_0^1 \frac{R(\theta)}{\theta(1-\theta)}d\theta
\end{align*}
The difference between the left-hand side and the right-hand side is at most $2\delta I(A,B)$, so taking the limit $\delta \to 0$ and using the squeeze theorem we get 
$$\lim_{c\to \infty} \frac{{\rm Tr} f(S_{cA,B})}{c^{d-1}\log(c)} = I(A,B)\int_0^1 \frac{f(\theta)}{\theta(1-\theta)}d\theta,$$
which establishes \eqref{area law two term} if $f$ is continuous and vanishes outside of $[\eps, 1-\eps]$.

Now, we turn to $f(\theta) = 1_{[a,b]}(\theta)$ for $0 < a < b < 1$. For small enough $\delta > 0$ consider the continuous functions $Q_\delta(\theta) = \max \left( 1-\frac{\operatorname{dist}(\theta, [a,b])}{\delta}, 0 \right)$ and $R_\delta(\theta) = \max \left( 1-\frac{\operatorname{dist}(\theta, [a+\delta,b-\delta])}{\delta}, 0 \right)$ which satisfy $Q_\delta(\theta) \geq 1_{[a,b]}(\theta) \geq R_\delta(\theta)$. Letting $\delta \to 0$ and applying the squeeze theorem again we deduce \eqref{area law two term} for $1_{[a,b]}$.

Now, we turn to the general $f:[0,1]\to \R$ with $f(1) = 0$ which is area law admissible and trace class admissible for $L^2(\R)$ and which is Riemann integrable on $[\eps, 1-\eps]$ for all $0 < \eps < \frac{1}{2}$. By the Riemann integrability on $[\eps,1-\eps]$ we can find Riemann upper and lower sums $Q_\eps \geq f \geq R_\eps$ such that
\begin{equation}\label{RQ difference}
\int_{\eps}^{1-\eps} \frac{Q_\eps(\theta) - R_\eps(\theta)}{\theta(1-\theta)} \, d\theta < \eps.
\end{equation}
Here $Q_\eps$ and $R_\eps$ are finite linear combinations of indicators of the form $1_{[a,b]}$ for $\eps \leq a <b \leq 1-\eps$. Define 
\[
G_{\eps}(\theta) = \begin{cases}
M_0f(\theta), & 0\le\theta<\eps,\\
0, & \eps \leq \theta \leq 1-\eps,\\
M_1 f(1-\theta), & 1-\eps < \theta \leq 1.
\end{cases}
\]
 We have the following chain of inequalities
\[
Q_{\eps}(\theta) + G_\eps(\theta) \geq f(\theta) \geq R_{\eps}(\theta) -G_\eps(\theta)
\]
and therefore, as before,  
\begin{equation*}
\begin{aligned}
\MoveEqLeft \limsup_{\eps\to 0^+} \left(\limsup_{c\to \infty}\frac{1}{c^{d-1}\log(c)}\operatorname{Tr}Q_{\eps}(S_{cA,B}) + \limsup_{c\to\infty}\frac{1}{c^{d-1}\log(c)}\operatorname{Tr} G_\eps(S_{cA,B})\right)\\
&\geq \limsup_{c\to \infty} \frac{1}{c^{d-1}\log(c)} \operatorname{Tr}f(S_{cA,B}) \geq \liminf_{c\to \infty} \frac{1}{c^{d-1}\log(c)} \operatorname{Tr} f(S_{cA,B})\\
&\geq \liminf_{\eps\to 0^+} \left(\liminf_{c\to \infty}\frac{1}{c^{d-1}\log(c)}\operatorname{Tr}R_{\eps}(S_{cA,B}) - \limsup_{c\to\infty}\frac{1}{c^{d-1}\log(c)}\operatorname{Tr} G_\eps(S_{cA,B})\right).
\end{aligned}
\end{equation*}
Since $Q_\eps$ and $R_\eps$ are finite linear combinations of indicators of the form $1_{[a,b]}$ for $0 < a < b < 1$, we know that 
\begin{align*}
\limsup_{c\to \infty}\frac{1}{c^{d-1}\log(c)}\operatorname{Tr}Q_{\eps}(S_{cA,B}) &= I(A,B)\int_0^1 \frac{Q_\eps(\theta)}{\theta(1-\theta)}\,d\theta,\\
\liminf_{c\to \infty}\frac{1}{c^{d-1}\log(c)}\operatorname{Tr}R_{\eps}(S_{cA,B}) &= I(A,B)\int_0^1 \frac{R_\eps(\theta)}{\theta(1-\theta)}\,d\theta.
\end{align*}
Taking the limit $\eps \to 0$, it follows from \eqref{RQ difference} that
\begin{align*}
\limsup_{\eps \to 0^+} \int_0^1 \frac{Q_\eps(\theta)}{\theta(1-\theta)}\,d\theta &= \int_0^1 \frac{f(\theta)}{\theta(1-\theta)}\, d\theta,\\
\liminf_{\eps \to 0^+} \int_0^1\frac{R_\eps(\theta)}{\theta(1-\theta)}\,d\theta &= \int_0^1 \frac{f(\theta)}{\theta(1-\theta)}\, d\theta.
\end{align*}
Thus, it suffices to establish that 
\begin{equation}\label{limsupG}\limsup_{\eps \to 0^+} \limsup_{c\to\infty} \frac{1}{c^{d-1}\log(c)}\operatorname{Tr} G_\eps(S_{cA,B}) = 0.
\end{equation}
Here we use the established bound \eqref{area law uniform bound}. It gives 
$$\limsup_{c\to \infty}\frac{1}{c^{d-1}\log(c)}\operatorname{Tr} G_\eps(S_{cA,B})  \leq C(A,B) \int_0^1 \frac{M_0G_\eps(\theta)+M_1G_\eps(\theta)}{\theta}d\theta.$$
For $\theta < \eps$ we have $M_0G_\eps(\theta) = M_0f(\theta)$ and $M_1G_\eps(\theta) = M_1f(\theta)$, thus $$M_0G_\eps(\theta)+M_1G_\eps(\theta) \le M_0f(\theta)+M_1f(\theta).$$ For $\eps \le \theta < 1$ we have 
\begin{equation}\label{Mbound}
M_0 G_\eps(\theta) + M_1G_\eps(\theta) \le 2\max(M_0f(\eps), M_1f(\eps)) \le 2(M_0f(\theta)+ M_1f(\theta)).
\end{equation}
We get that $\frac{M_0 G_\eps(\theta) + M_1G_\eps(\theta)}{\theta(1-\theta)}$ is majorized by the $L^1$-function $\frac{2(M_0f(\theta)+ M_1f(\theta))}{\theta(1-\theta)}$. Thus, if we show that it tends to zero pointwise as $\eps \to 0^+$ then the dominated convergence theorem will imply \eqref{limsupG}. By \eqref{Mbound} it suffices to show that 
\begin{equation}\label{cont at 0}
\max(M_0f(\eps), M_1f(\eps))\to 0 \text{ as } \eps \to 0^+.
\end{equation}
Clearly, $\max(M_0f(\eps), M_1f(\eps))$ is non-negative and non-decreasing in $\eps$, thus the right limit $\lim_{\eps \to 0^+}\max(M_0f(\eps), M_1f(\eps)) = \nu$ exists and is non-negative. If $\nu > 0$ then the integral
$$\int_0^1 \frac{M_0f(\theta)+M_1f(\theta)}{\theta}d\theta$$
diverges, contradicting the assumption that $f$ is area law admissible. Thus, $\nu = 0$ giving us \eqref{limsupG}.

\section{Computation of $\operatorname{Tr} S_{A,B}^2$ when $A$ and $B$ are finite unions of boxes}
In this section we prove Theorem \ref{TrS2} and Theorem \ref{TrS2 explicit}. The key idea of the computation is that $\operatorname{Tr} S_{A,B}^2 = \|S_{A,B}\|_{2}^2$ is the Hilbert--Schmidt norm squared of $S_{A,B}$. Since $S_{A,B}$ is an operator with the kernel \eqref{kernel}, we have
$$\|S_{A,B}\|_2^2 = \iint_{\R^d\times \R^d}|1_{A}(x) \check{1}_B(x-y) 1_A(y)|^2\,dx\, dy.$$

Assume that $A = \cup_{k=1}^n A_k$ and $B = \cup_{l=1}^m B_l$ are finite unions of axis-parallel boxes with disjoint interiors. We have
\begin{align*}
\MoveEqLeft\iint_{\R^d\times \R^d}|1_{A}(x) \check{1}_B(x-y) 1_A(y)|^2\,dx\, dy \\
&=\sum_{k_1,k_2=1}^n \sum_{l_1, l_2=1}^m \iint_{\R^d\times \R^d} 1_{A_{k_1}}(x) 1_{A_{k_2}}(y)\check{1}_{B_{l_1}}(x-y)\check{1}_{B_{l_2}}(y-x)\, dx\, dy.
\end{align*}
For each $k_1, k_2, l_1, l_2$ the variables separate and we reduce to the integrals of the form
$$\iint_{\R\times\R} 1_{I_1}(x)1_{I_2}(y)\check{1}_{J_1}(x-y)\check{1}_{J_2}(y-x) dx\, dy$$
for some intervals $I_1, I_2, J_1, J_2\subset \R$. Doing the change of variables $x-y = z$ and applying Fubini's theorem we get
\begin{equation}\label{convolution stupid}
\iint_{\R\times\R} 1_{I_1}(x)1_{I_2}(x+z)\check{1}_{J_1}(z)\check{1}_{J_2}(-z) dx\, dz = \int_{\R} \check{1}_{J_1}(z)\check{1}_{J_2}(-z) |I_1\cap (I_2-z)| dz.
\end{equation}
We do exactly the same on the Fourier side with the intervals $J_1$ and $J_2$. For each $z\in \R$
\begin{align*}
\check{1}_{J_1}(z)\check{1}_{J_2}(-z) = \iint_{\R\times \R} 1_{J_1}(u) 1_{J_2}(v) e^{2\pi i z(u-v)} \, du \, dv = \int_{\R} e^{2\pi i z w} |J_1 \cap (J_2 - w)| \, dw. 
\end{align*}
Plugging this into \eqref{convolution stupid} we get
\begin{equation}\label{feasible}
\iint_{\R\times\R} e^{2\pi i zw}|I_1\cap (I_2-z)||J_1\cap (J_2-w)| \,dz\, dw
\end{equation}
For fixed intervals $I, I'$ the function $\R \ni t\to I\cap (I'-t)$ is zero up to some point $p_1$, then linear up to some $p_2$, then constant up to some $p_3$, linear again up to some $p_4$ and zero afterwards (if $|I| = |J|$ then $p_2 = p_3$). Thus, separating into the nine cases with respect to the pairs $I_1, I_2$ and $J_1, J_2$ we have to compute the integrals of the form
$$\int_a^b\int_c^d (\alpha z+\beta)(\gamma w + \delta)e^{2\pi i zw}dzdw.$$
A direct computation shows 
\begin{align*}
\MoveEqLeft\int_a^b\int_c^d (\alpha z+\beta)(\gamma w + \delta)e^{2\pi i zw}\, dz\, dw\\
&= \int_a^b(\gamma w+\delta)\left(\frac{(\alpha d+\beta)e^{2\pi i w d}-(\alpha c+\beta)e^{2\pi i wc}}{2\pi i w}
+\frac{\alpha\big(e^{2\pi i wd}-e^{2\pi i wc}\big)}{4\pi^2 w^2}\right)dw.
\end{align*}
This integral can be computed explicitly with the use of the exponential integral function $E_1(w)$, because $\int \frac{e^{ir}}{r^2}dr = -\frac{e^{ir}}{r} + i\int \frac{e^{ir}}{r}$. If $0\in [a, b]$ then one has to exercise a bit of care, removing a small interval $[-\eps, \eps]$, integrating over the resulting segments and letting $\eps \to 0$ with the use of the known asymptotics of $E_1(w)$ for small $w$ which will give the logarithmic terms.

Unfortunately, executing this strategy in practice is rather infeasible due to the number of terms appearing, so we will only do the simplest one-dimensional case $A = [0, c], B = [0,1]$. Our starting point will be \eqref{feasible} with $I_1 = I_2 = [0,c]$ and $J_1 = J_2 = [0,1]$. We find that
\begin{align*}
\operatorname{Tr} S_{[0,c],[0,1]}^2 &= \iint_{\R\times\R} e^{2\pi i zw}\,\left|[0,c]\cap ([0,c]-z)\right|\,\left|[0,1]\cap ([0,1]-w)\right| \,dz\, dw\\
& = \int_{-1}^1 \int_{-c}^c  e^{2\pi i z w} \left( c - |z|\right) \left( 1 - |w|\right) dz \, dw.
\end{align*}
Using Fubini's theorem and taking the integral in $w$ we get
$$\operatorname{Tr} S_{[0,c],[0,1]}^2 = \int_{-c}^c (c-|z|)\frac{\sin(\pi z)^2}{\pi^2 z^2}dz = 2\int_0^c (c-z)\frac{\sin(\pi z)^2}{\pi^2z^2}dz = \lim_{\eps \to 0^+} 2\int_\eps^c (c-z)\frac{\sin(\pi z)^2}{\pi^2z^2}dz. $$
Using $\sin^2(t) = \frac{1-\cos(2t)}{2}$ and $\int \frac{\cos(t)}{t^2} = -\frac{\cos(t)}{t} - \int \frac{\sin(t)}{t}$ we can find an explicit primitive and get
$$2\int_\eps^c (c-z)\frac{\sin(\pi z)^2}{\pi^2z^2}dz = \frac{2\pi cz                 {\rm Si}(2\pi z)+c\cos(2\pi z)-c+z{\rm Ci}(2\pi z)-z\log(z)}{\pi^2 z} \Bigg \rvert_\eps^c.$$
The limit of the primitive as $\eps \to 0$ is $\frac{\gamma + \log(2\pi)}{\pi^2}$ using the known asymptotics of ${\rm Ci}(z)$ for small $z$. So, 
$$\operatorname{Tr} S_{[0,c],[0,1]}^2 = c - \frac{\log(c)}{\pi^2} - \frac{1+\gamma + \log(2\pi)}{\pi^2} + \left(\frac{2}{\pi} c\left({\rm Si}(2\pi c)-\frac{\pi}{2}\right) + \frac{\cos(2\pi c)}{\pi^2} + \frac{{\rm Ci}(2\pi c)}{\pi^2}\right),$$
as required. Recall the known asymptotics for ${\rm Si}(t)$ and ${\rm Ci}(t)$ for large $t$
\begin{equation}\label{Si asymp}
{\rm Si}(t) = \frac{\pi}{2} - \cos(t) \sum_{n=0}^\infty \frac{(-1)^n (2n)!}{t^{2n+1}}+ \sin(t) \sum_{n=1}^\infty \frac{(-1)^n(2n-1)!}{t^{2n}},
\end{equation}
\begin{equation}\label{Ci asymp}
{\rm Ci}(t) = \sin(t) \sum_{n=0}^\infty \frac{(-1)^n(2n)!}{t^{2n+1}} + \cos(t)\sum_{n=1}^\infty \frac{(-1)^n(2n-1)!}{t^{2n}}.
\end{equation}
We note that the first term in the first sum in \eqref{Si asymp} cancels with $\frac{\cos(2\pi c)}{\pi^2}$ and the first term in the second sum in \eqref{Si asymp} cancels with the first term in the first sum in \eqref{Ci asymp}. Thus, the oscillating terms begin with $O(\frac{1}{c^2})$ term. Plugging these series into our formula gives the desired asymptotic expression.

\section{One-term asymptotics}

To establish one-term asymptotic \eqref{one-term} in general we first need to establish it for any two qualitatively different functions directly. For $f(\theta) = \theta$ we already know it, even without any error terms, so we will consider $f(\theta) = \theta^2$.

\begin{claim}\label{the trivial claim}
Let $A, B \subseteq \R^d$ be sets with finite measure. Then $S_{cA,B}^2$ is trace class and 
\[
\operatorname{Tr}S_{cA,B}^2 = c^d |A||B| + o(c^d) 
\]
as $c\to \infty$.
\end{claim}
\begin{proof}
It suffices to show that $\operatorname{Tr} \big[S_{cA,B}- S_{cA,B}^2\big] = o(c^d)$. By the $S-S^2$ trick we have
\[
S_{cA,B}- S_{cA,B}^2= \left|P_{cA} Q_B P_{cA^c}\right|^2,
\]
so the trace is given by the Hilbert--Schmidt norm squared of $P_{cA} Q_B P_{cA^c}$. Arguing like in the previous section, we find
\begin{align*}
\norm{P_{cA} Q_B P_{cA^c}}_2^2 = c^d \int_{\R^d} \left| \check{1}_B(x)\right|^2 \left|A \cap \left(A^c - \frac{x}{c}\right)\right| \, dx.
\end{align*}
The function $F_A(x) = |A \cap (A^c -x)| = 1_A \ast 1_{-A^c}(-x)$ is continuous, bounded, and satisfies $F_A(0) = 0$. The claim now follows immediately from the dominated convergence theorem. 
\end{proof}

\begin{proof}[Proof of Theorem \ref{one term general}]
Subtracting $\theta f(1)$ from $f$ we can without loss of generality assume that $f(1) = 0$. Pick a small $\frac{1}{2} > \eps > 0$. We split ${\rm Tr}f(S_{cA,B})$ into the terms with $\lambda_n > 1-\eps$ and $1-\eps \ge \lambda_n \ge 0$:
\begin{align*}
\MoveEqLeft{\rm Tr}f(S_{cA,B})= \sum_{\lambda_n(cA, B) > 1-\eps}f(\lambda_n(cA,B))  +\sum_{1-\eps\ge\lambda_n(cA, B) \geq 0 }f(\lambda_n(cA,B)).
\end{align*}
For the second sum we simply apply the linear bound $|f(\theta)| \leq C \theta$ and the elementary inequality $\theta \leq \frac{\theta(1-\theta)}{\eps}$, valid for $0\leq \theta \leq 1-\eps$. We find
\[
\left|\sum_{1-\eps\ge\lambda_n(cA, B) \geq 0 }f(\lambda_n(cA,B))\right| \leq \frac{C}{\eps} \operatorname{Tr}\big[ S_{cA,B} - S_{cA,B}^2\big],
\]
which is $o(c^d)$ for all fixed $0<\eps <\frac{1}{2}$ by Claim \ref{the trivial claim}.

For the first sum, by continuity of $f$ at $1$, for any $\delta > 0$ there is $\eps$ small enough such that $|f(\lambda_n(cA, B))| < \delta$ if $\lambda_n(cA, B) > 1-\eps$. Since $0 < \eps < \frac{1}{2}$, this implies that $|f(\lambda_n(cA,B)| < 2\delta \lambda_n(cA,B)$. Thus, the first sum is at most $2\delta c^d |A| |B|$. Since $\delta >0$ can be arbitrarily small, this gives us the result.
\end{proof}

Now, we turn to the proof of Theorem \ref{one term bounded}. This time, we will do a more complicated splitting of ${\rm Tr}f(S_{cA,B})$. 

\begin{proof}[Proof of Theorem \ref{one term bounded}]
As in the previous proof we can without loss of generality assume that $f(1) = 0$. We also remark that since $f$ is trace class admissible for $L^2(\R^d)$, $f$ must satisfy $\lim_{\theta\to 0^+} f(\theta) =0$, by an argument similar to the proof of \eqref{cont at 0}, just with the trace class admissible for $L^2(\R^d)$ condition instead of the area law admissible condition.

Since $A, B$ are bounded, there exist boxes $V_A, V_B$ such that $A\subseteq V_A$, $B\subseteq V_B$. Fix once and for all a large number $D$ to be determined later (it will only depend on $V_A, V_B$). Let $0 < \eps < \frac{1}{2}$ be a small number. We have
\begin{align*}
{\rm Tr}f(S_{cA,B}) &= \sum_{\lambda_n(cA, B) > 1-\eps}f(\lambda_n(cA,B))  +\sum_{1-\eps\ge\lambda_n(cA, B) \geq \eps }f(\lambda_n(cA,B))\\
&\quad +\sum_{\substack{0 \le \lambda_n(cA,B)< \eps \\ n\le Dc^d}} f(\lambda_n(cA,B))+ \sum_{\substack{0 \le \lambda_n(cA,B)< \eps\\ n> Dc^d}} f(\lambda_n(cA,B)).
\end{align*}
For the first sum for any $\delta > 0$ we can choose $\frac{1}{2}>\eps > 0$ small enough so that it is at most $2\delta |A| |B| c^d$, as in the previous proof. 

For the second sum we simply use that $f$ is bounded and the elementary inequality $1\leq \frac{\theta(1-\theta)}{\eps(1-\eps)}$ for $\theta\in [\eps, 1-\eps]$. Thus, the second sum is at most $\frac{ {\rm Tr}[S_{cA,B}-S_{cA,B}^2]}{\eps(1-\eps)}\sup_{t\in [0,1]} |f(t)|$, which is $o(c^d)$ by Claim \ref{the trivial claim}. 

For the third sum, we use that $\lim_{\theta\to 0^+} f(\theta) =0$ to conclude that for any $\delta > 0$ we can choose $\eps > 0$ small enough so that it is at most $\delta Dc^d$.

Finally, we turn to the most challenging fourth sum. Firstly, we estimate $|f(\lambda_n(cA, B))|$ by $M_0f(\lambda_n(cA,B))$. Since $M_0f$ is non-decreasing, we can upper bound $\lambda_n(cA,B)$ by $\lambda_n(cV_A, V_B)$ using Lemma \ref{enlargement}. So, it remains to show that
$$\sum_{n > Dc^d} M_0f(\lambda_n(cV_A, V_B)) = o(c^d).$$
We claim that if $D$ is large enough and $n> D c^d$, then $\lambda_n(cV_A, V_B) \le \alpha_d^{-c}$, where $\alpha_d$ is taken from Theorem \ref{two cubes}. This is equivalent to saying that $N_{\alpha_d^{-c}}(cV_A, V_B)\le Dc^d$. We have
$$N_{\alpha_d^{-c}}(cV_A, V_B) = N_{1/2}(cV_A,V_B) + \Lambda_{\alpha_d^{-c}}(cV_A,V_B).$$
The first term is $O(c^d)$ by \eqref{N 1/2} and the second term is $O(c^d)$ by \eqref{Lambda -, equivalence, two cubes}. Thus, $N_{\alpha_d^{-c}}(cV_A, V_B)\le Dc^d$ holds for large enough $D$. Given the claim, the computation \eqref{trace class condition bound} therefore shows that
\[
\sum_{n > Dc^d} M_0f(\lambda_n(cV_A, V_B)) \leq \sum_{\lambda_n(cV_a,V_B) \leq \alpha_d^{-c}}  M_0f(\lambda_n(cV_A, V_B)) = O(\log(c)^d) = o(c^d),
\]
which finishes the proof.
\end{proof}

\subsection{Counterexample for unbounded sets}
In this subsection we prove Proposition \ref{one term counterexample}. First, for convenience we replace $f$ with a monotone subordinate function $g$.
\begin{claim}
For any function $f:[0,1]\to \Cm$ such that $\lim_{x\to 0^+} \frac{|f(x)|}{x} = \infty$ there exist $\eps > 0$ and $g:[0,1]\to [0, \infty)$ such that $g$ is non-decreasing, $\lim_{x\to 0^+} \frac{g(x)}{x} = \infty$ and $g(x)\le |f(x)|$ for $0 \le x 
\le\eps$.
\end{claim}
\begin{proof}
Since $\lim_{x\to 0^+} \frac{|f(x)|}{x} = \infty$, for any $k\in \N$ there exists $y_k > 0$ such that for $0 < x \le y_k$ we have $|f(x)| \ge 2^k x$. Inductively making $y_k$ smaller if necessary we can additionally assume that $y_k \ge 2 y_{k+1}$. We define $g(y_k) = 2^{k-1}y_k$, between $y_k$ and $y_{k+1}$ we extend $g$ linearly, and for $x\in [y_1, 1]$ we set $g(x) = g(y_1)$. We are going to show that $g$ is non-decreasing, $\lim_{x\to 0^+} \frac{g(x)}{x} = \infty$ and $g(x)\le |f(x)|$ for $0 < x \le y_1$.

For the first assertion it is enough to check that $g(y_k) \ge g(y_{k+1})$ since linear functions do not change monotonicity. This is equivalent to $y_k \ge 2y_{k+1}$ which is true by our assumption. For the second assertion, we have $$\inf_{x\in [y_k, y_{k+1}]} \frac{g(x)}{x} = \frac{g(y_k)}{y_k}=2^{k-1},$$
which tends to infinity as $k \to \infty$, thus $\lim_{x\to 0^+} \frac{g(x)}{x} = \infty$. Analogously, we also have 
$$\sup_{x\in [y_k, y_{k+1}]} \frac{g(x)}{x} = \frac{g(y_{k+1})}{y_{k+1}}=2^{k},$$
thus $g(x) \le 2^{k}x$ for $x\in [y_k, y_{k+1}]$. On the other hand, for $x\in [y_k, y_{k+1}]$ we have $|f(x)| \ge 2^k x$, which gives us the last assertion. Finally, we put $g(0) = 0$ to complete the proof.
\end{proof}
For any sets $A, B$ of finite measure we clearly have that if $f(S_{A,B})$ is trace class then $g(S_{A, B})$ is trace class. Thus, given a set $B$ of positive and finite measure, it suffices to construct a set $A$ such that $g(S_{A,B})$ is not trace class, or equivalently $\operatorname{Tr} g(S_{cA,B}) = \infty$.

We will take $A = \bigcup_{k\in \N} A_k$ where each $A_k$ has measure $\frac{1}{2^k}$. By monotonicity, ${\rm Tr}\,g(S_{A,B}) \ge {\rm Tr}\,g(S_{A_k, B})$, thus it suffices to make ${\rm Tr}\,g(S_{A_k, B})$ tend to infinity. Since $\lim_{x\to 0^+} \frac{g(x)}{x} = \infty$, for each $k$ there exists $x_k > 0$ such that $g(x) \ge 4^k x$ for $0 < x \le x_k$. If $S_{A_k, B}$ does not have eigenvalues larger than $x_k$ then we have
$${\rm Tr}\,g(S_{A_k, B}) = \sum_{n=1}^\infty g(\lambda_n(A_k, B)) \ge \sum_{n=1}^\infty 4^k \lambda_n(A_k, B) = \frac{4^k}{2^k}|B| = 2^k|B|,$$
which tends to infinity as $k\to \infty$.

So, it remains to construct sets $A_k$ of measure $\frac{1}{2^k}$ such that $S_{A_k, B}$ does not have eigenvalues larger than $x_k$. This is equivalent to demanding that $\lambda_1(A_k, B) \le x_k$. We clearly have 
$$\lambda_1(A_k, B) \le \left(\sum_{n=1}^\infty \lambda_n(A_k,B)^2\right)^{1/2} = \left({\rm Tr}S_{A_k,B}^2\right)^{1/2}.$$
Thus, if we can make ${\rm Tr}S_{A_k,B}^2$ as small as we like, we will get the desired inequality. We pick a large number $N$ and let $A_k$ be a union of $N$ vastly separated intervals of length $\frac{1}{N2^k}$:
$$A_k = \bigcup_{m=1}^N \left[Nm, Nm + \frac{1}{N2^k}\right].$$

We have 
$${\rm Tr}S_{A_k,B}^2 = \int_\R |\check{1}_B(x)|^2 \left|A_k \cap (A_k - x)\right|\,dx.$$
Direct inspection shows that for $\frac{1}{N2^k}< |x| < N-\frac{1}{N2^k}$ we have $A_k \cap (A_k - x) = \varnothing$ and for $x$ outside this range the measure of the intersection is clearly at most $\frac{1}{2^k}$. Thus, 
$${\rm Tr}S_{A_k,B}^2 \le \frac{1}{2^k}\int_{|x| < \frac{1}{N2^k}} |\check{1}_B(x)|^2dx + \frac{1}{2^k}\int_{|x| >N- \frac{1}{N2^k}} |\check{1}_B(x)|^2dx.$$
As $N\to \infty$, by the dominated convergence theorem both of these integrals go to $0$. Thus, we can make ${\rm Tr}S_{A_k,B}^2$ as small as we like by taking $N$ large enough, as required.

\subsection*{Acknowledgments} Aleksei Kulikov and Martin Dam Larsen were supported by the VILLUM Centre of Excellence for the Mathematics of Quantum Theory (QMATH) with Grant No.10059.

\bibliography{ref}

\end{document}